\documentclass[11pt,           
draft,                         
english                        
]{article}

\usepackage[english,strings]{babel}     
\usepackage[utf8]{inputenc}    
\usepackage[T1]{fontenc}       
\usepackage{longtable}         
\usepackage{exscale}           
\usepackage[final]{graphicx}   
\usepackage[sort]{cite}        
\usepackage{array}             
\usepackage{fancyhdr}          
\usepackage[a4paper]{geometry} 
\usepackage{xspace}            
\usepackage{tikz}              
\usepackage{tikzsymbols}       
\usepackage{nchairx}           
\usepackage{wasysym}           
\usepackage[sectionbib         
           ]{chapterbib}       
\usepackage{appendix}
\usepackage[expansion=false    
           ]{microtype}        
\usepackage[nottoc]{tocbibind} 
\usepackage[backref=page,      
           final=true,         
           pdfpagelabels       
           ]{hyperref}         

\usepackage{tikz-cd}       
\usepackage{xcolor} 
\usetikzlibrary{arrows,calc,through,backgrounds,matrix,
	decorations.pathmorphing,positioning,babel}    

\newcommand{\chairxauthorbibfont}{\textsc}

\newcommand{\chairxtitlebibfont}{\textit}

\newcommand{\chairxseriesbibfont}{\textit}

\newcommand{\poly}{{\scriptscriptstyle{\mathrm{poly}}}}    
\newcommand{\Tpoly}{T_\poly}  
\newcommand{\Dpoly}{D_\poly}  
\newcommand{\hkr}{\mathsf{hkr}}   
\newcommand{\Tay}{{\scriptscriptstyle{\mathrm{Tay}}}}    
 
\newcommand{\nice}{{\scriptscriptstyle{\mathrm{nice}}}}      

\title{The Strong Homotopy Structure of BRST Reduction}
\author{
  \textbf{Chiara Esposito}\thanks{\texttt{chesposito@unisa.it}},\\[0.3cm]
   Dipartimento di Matematica\\
   Università degli Studi di Salerno\\
   via Giovanni Paolo II, 123\\
   84084 Fisciano (SA), Italy \\[0.5cm]
   \textbf{Andreas Kraft}\thanks{\texttt{akraft@impan.pl}},\\[0.3cm]
	 Institute of Mathematics  \\
	  Polish Academy of Sciences \\
		ul. \'Sniadeckich 8 \\ 
		00-656 Warsaw,
		Poland  \\[0.5cm]
  \textbf{Jonas Schnitzer}\thanks{\texttt{jonas.schnitzer@math.uni-freiburg.de}},\\[0.3cm]
  Department of Mathematics\\
   University of Freiburg\\
   Ernst-Zermelo-Straße, 1\\
	D-79104 Freiburg, Germany
}
\date{\today}

\begin{document}

\maketitle

\abstract{
In this paper we propose a reduction scheme for polydifferential operators 
phrased in terms of $L_\infty$-morphisms. The desired reduction $L_\infty$-morphism has been obtained
by applying an explicit version of the homotopy transfer theorem.
Finally, we prove that the reduced star product
induced by this reduction $L_\infty$-morphism and the reduced star product
obtained via the formal Koszul complex are equivalent.
	}
		
	\bigskip


\newpage

\tableofcontents

\section{Introduction}

This paper aims to propose a reduction scheme for equivariant 
polydifferential operators that is phrased in terms of 
$L_\infty$-morphisms, generalizing the results from 
\cite{esposito.kraft.schnitzer:2020a:pre}, obtained for
polyvector fields. Our main motivation comes from formal deformation 
quantization: 
Deformation quantization has been introduced by
Bayen, Flato, Fronsdal, Lichnerowicz and Sternheimer in \cite{bayen.et.al:1978a} and it relies on the
idea that the quantization of a phase space described by a Poisson manifold $M$
is described by a formal deformation, so-called \emph{star product}, of the commutative algebra of smooth
complex-valued functions $\Cinfty(M)$ in a formal parameter $\hbar$.
The existence and classification of 
star products on Poisson manifolds has been provided by Kontsevich's
formality theorem \cite{kontsevich:2003a}, whereas the invariant 
setting of Lie group actions has been treated by 
Dolgushev, see~\cite{dolgushev:2005a,dolgushev:2005b}. 
More explicitly, the formality theorem provides an $L_\infty$-quasi-isomorphism between 
the differential graded Lie algebra (DGLA) of polyvector fields 
$\Tpoly(M)$ and the polydifferential operators $\Dpoly(M)$ 
resp. the invariant versions. 
As such, it maps in particular Maurer-Cartan elements in the DGLA of polyvector 
fields, i.e. (formal) Poisson structures, to Maurer-Cartan elements 
in the DGLA of polydifferential operators, which correspond to star products.
 
One open question and our main motivation is 
to investigate the compatibility of deformation quantization and 
phase space reduction in the Poisson setting, and in this present paper 
we propose a way to describe the reduction on the quantum side by an 
$L_\infty$-morphism. Given a Lie group 
$\group{G}$ acting on a manifold $M$, we aim to reduce \emph{equivariant star products} $(\star,H)$, i.e. 
pairs consisting of an invariant star product $\star$ and a 
quantum momentum map $H= \sum_{r=0}^\infty \hbar^r J_r \colon \liealg{g} \longrightarrow \Cinfty(M)[[\hbar]]$, where $\liealg{g}$ is the Lie algebra of
$\group{G}$. In this case, $J_0$ is a classical momentum map for the Poisson structure induced by $\star$. Interpreting it as smooth map 
$J_0\colon M\longrightarrow \liealg{g}^*$ and assuming that 
$0\in \liealg{g}^*$ is a value and regular value, it follows 
that $C = J^{-1}(\{0\})$ is a closed embedded submanifold of
$M$ and by the Poisson version of the Marsden-Weinstein reduction 
\cite{marsden.weinstein:1974a} we know that unders suitable assumptions
the reduced manifold $M_\red = C/\group{G}$ is again a Poisson 
manifold if the action on $C$ is proper and free. 
In this setting, there is a well-known BRST-like reduction procedure \cite{bordemann.herbig.waldmann:2000a,gutt.waldmann:2010a} of equivariant star products on $M$ to star products on $M_\red$.

In order to describe this reduction by an $L_\infty$-morphism, we have 
to fix at first the DGLA controlling Hamiltonian actions in the 
quantum setting, i.e. a DGLA whose Maurer-Cartan elements correspond to 
equivariant star products. We denote it by 
\begin{equation*}
  (D_{\liealg{g}}(M)[[\hbar]],\hbar\lambda,
	\del^\liealg{g}-[J_0,\argument]_\liealg{g}, 
	[\argument,\argument]_\liealg{g}),
\end{equation*}
where $\lambda = e^i \otimes (e_i)_M$ is given by the fundamental vector 
fields of the $\group{G}$-action in terms of a basis $e_1,\dots,e_n$ of 
$\liealg{g}$ with dual basis $e^1,\dots,e^n$ of $\liealg{g}^*$. 
It is called the DGLA of \emph{equivariant polydifferential operators}. 

The construction of the desired $L_\infty$-morphism to 
$(\Dpoly(M_\red),\del,[\argument,\argument]_G)$ is then based on 
the following steps:
\begin{itemize}
  \item Assuming for simplicity $M=C\times\liealg{g}^*$, which always holds 
	      locally in suitable situations, we can perform a Taylor expansion 
				around $C$ and end up with a DGLA $D_\Tay(C\times\liealg{g}^*)$. 
				Using a 'partial homotopy', we find a deformation retract to 
				a DGLA structure on the space $ \left( \prod_{i=0}^\infty( \Sym^i \liealg{g} \otimes \Dpoly(C))\right)^\group{G}$, i.e. we get rid of 
				differentiations in $\liealg{g}^*$-direction.
	\item For the polyvector fields in \cite{esposito.kraft.schnitzer:2020a:pre} 
				we used the canonical linear 
	      Poisson structure $\pi_\KKS $ on the dual of the action Lie algebroid 
				$C \times \liealg{g}$ for the reduction.
				The analogue structure in our quantum setting is the product on 
				the quantized universal enveloping algebra $\Universal_\hbar(C\times\liealg{g})$ 
				of the action Lie algebroid. We use this product to perturb the 
				deformation retract from the last point. This is more complicated 
				as the polyvector field case since we have to use now the homological 
				perturbation lemma to perturb the involved chain maps, and the 
				deformed maps are no longer compatible with the Lie brackets.
				
	\item We use the homotopy transfer theorem to construct the 
	      $L_\infty$-projection from the Taylor expansion to 
				$ \left( \prod_{i=0}^\infty( \Sym^i \liealg{g} \otimes \Dpoly(C))\right)^\group{G}$ with transferred 
				$L_\infty$-structure. Notice that in the 
				polyvector field case it was not necessary to 
				transfer the DGLA structure. 
				
	\item We check in Proposition~\ref{prop:HigherBracketsDtayvanish} that 
	      the transferred $L_\infty$-structure is just a DGLA structure, 
				and in Proposition~\ref{prop:prProjectionfromTransferredStructure} 
				that the transferred Lie bracket is compatible with the 
				projection to $D_\poly(M_\red)[[\hbar]]$. Thus we get the 
	      reduction $L_\infty$-morphism from the Taylor expansion 
				to the polydifferential 
				operators on $M_\red$. Twisting it by the product on the universal 
				enveloping algebra ensures that we start in the right curved 
				DGLA structure. 
\end{itemize}

Finally, the morphism can be globalized to general smooth manifolds 
$M$ with sufficiently nice Lie group actions 
and we get the following result (Theorem~\ref{thm:GlobalDred}):

\begin{nntheorem}
  There exists an $L_\infty$-morphism 
\begin{equation}
\label{eq:DredIntro}
  \mathrm{D}_\red \colon
  (D_{\liealg{g}}(M)[[\hbar]],\hbar\lambda,\del^\liealg{g}-[J_0,\argument]_\liealg{g}, 
	[\argument,\argument]_\liealg{g})
	\longrightarrow 
	(\Dpoly(M_\red)[[\hbar]],0,
\del,[\argument,\argument]_G),
\end{equation}	   
	called \emph{reduction $L_\infty$-morphism}. 
\end{nntheorem}
Finally, we compare the reduction of equivariant star products 
via $\mathrm{D}_\red$ to a slightly modified version of 
the BRST reduction from 
\cite{bordemann.herbig.waldmann:2000a,gutt.waldmann:2010a}, see 
Theorem~\ref{thm:CompRed}:

\begin{nntheorem}
  Let $(\star,H)$ be an equivariant star product on $M$. 
	Then the reduced star product induced by 
	$\mathrm{D}_\red$ from \eqref{eq:DredIntro} and the 
	reduced star product 	via the formal Koszul complex are equivalent. 
\end{nntheorem}

Note that together with \cite[Theorem~5.1]{esposito.kraft.schnitzer:2020a:pre}
we have now the diagram:
\begin{equation*}
\begin{tikzcd}
(T^\bullet_{\liealg{g}}(M)[[\hbar]],\hbar\lambda,[-J_0,\argument]_\liealg{g},
[\argument,\argument]_\liealg{g}) 
\arrow[dd, "\mathrm{T}_\red"] 
&& (D^\bullet_{\liealg{g}}(M)[[\hbar]],\hbar\lambda,\del^\liealg{g}-[J_0,\argument]_\liealg{g}, 
		[\argument,\argument]_\liealg{g})
	\arrow[dd, "\mathrm{D}_\red"]\\ 
&  &\\
(\Tpoly^\bullet(M_\red)[[\hbar]],0,0,[\argument,\argument]_S)\arrow[rr, "F_\red"]& 
& (\Dpoly^\bullet(M_\red)[[\hbar]],0,
\del,[\argument,\argument]_G),
\end{tikzcd}
\end{equation*}
where $F_\red$ is the standard Dolgushev formality with 
respect to a torsion-free covariant derivative on $M_\red$. 
Moreover, in \cite{esposito.kraft.schnitzer:2022a} we show that 
the Dolgushev formality is compatible with $\lambda$ under 
suitable flatness assumptions. In these flat cases it 
induces an $L_\infty$-morphism
\begin{equation*}
    F^\liealg{g} \colon
    (T^\bullet_{\liealg{g}}(M)[[\hbar]],\hbar\lambda,[-J_0,\argument]_\liealg{g},[\argument,\argument]_\liealg{g})
		\longrightarrow
		(D^\bullet_{\liealg{g}}(M)[[\hbar]],\hbar\lambda,\del^\liealg{g}-[J_0,\argument]_\liealg{g}, 
		[\argument,\argument]_\liealg{g}),
\end{equation*}
which gives the forth arrow in the above diagram, and we plan to investigate 
its commutativity (up to homotopy) in future work.

The results of this paper are partially based on \cite{kraft:2021a} and 
the paper is organized as follows: In Section~\ref{sec:Preliminaries} we 
we recall the basic notions of (curved) $L_\infty$-algebras, 
$L_\infty$-morphisms and twists and fix the notation. Then we introduce in 
Section~\ref{sec:EquivDpoly} the curved DGLA of 
equivariant polydifferential operators 
and show that they control indeed Hamiltonian actions. In 
Section~\ref{sec:RedDpoly} we construct the global reduction 
$L_\infty$-morphism to the polydifferential operators on the reduced 
manifold. Finally, we compare in Section~\ref{sec:ComparisonofRedProd} 
the reduction via this reduction morphism $\mathrm{D}_\red$ with a slightly modified 
BRST reduction of equivariant 
star products as explained in Appendix~\ref{sec:BRSTReductionStarProducts}, 
where we also recall the homological perturbation lemma. 
In Appendix~\ref{sec:homotopytransfertheorem} 
we give explicit formulas for 
the transferred $L_\infty$-structure and the $L_\infty$-projection 
induced by the homotopy transfer theorem.

\subsection*{Acknowledgements}
  The authors are grateful to Ryszard Nest and Boris Tsygan 
	for the helpful comments. This work was supported by the 
	National Group for Algebraic and Geometric Structures, and 
	their Applications (GNSAGA – INdAM).
  The third author is supported by the DFG research training group 
  "gk1821: Cohomological Methods in Geometry".

\section{Preliminaries}
\label{sec:Preliminaries}

\subsection{$L_\infty$-Algebras, Maurer-Cartan Elements and Twisting}

In  this section we recall the notions of (curved) $L_\infty$-algebras,
$L_\infty$-morphisms and their twists by Maurer--Cartan elements to
fix the notation. 
Proofs and further details can be found in
\cite{dolgushev:2005a,dolgushev:2005b,esposito.dekleijn:2021a}.

We denote by $V^\bullet$ a graded vector space over a field $\mathbb{K}$ of
characteristic $0$ and define the \emph{shifted} vector space
$V[k]^\bullet$ by
\begin{equation*}
  V[k]^\ell
  =
  V^{\ell+k}.
\end{equation*}
A degree $+1$ coderivation $Q$ on the coaugmented counital conilpotent
cocommutative coalgebra $S^c(\mathfrak{L})$ cofreely cogenerated by
the graded vector space $\mathfrak{L}[1]^\bullet$ over $\mathbb{K}$ is
called an \emph{$L_\infty$-structure} on the graded vector space
$\mathfrak{L}$ if $Q^2=0$. The (universal) coalgebra
$S^c(\mathfrak{L})$ can be realized as the symmetrized
deconcatenation coproduct on the space
$\bigoplus_{n\geq0}\Sym^n\mathfrak{L}[1]$ where $\Sym^n\mathfrak{L}[1]$ 
is the space of coinvariants for the
usual (graded) action of $S_n$ (the symmetric group in $n$ letters) on
$\otimes^n(\mathfrak{L}[1])$, see e.g. \cite{esposito.dekleijn:2021a}. 
Any degree $+1$ coderivation $Q$ on $S^c(\mathfrak{L})$ is uniquely determined by the
components
\begin{equation}
  Q_n\colon \Sym^n(\mathfrak{L}[1])\longrightarrow \mathfrak{L}[2]
\end{equation}
through the formula 
\begin{align}
\begin{split}
  &  Q(\gamma_1\vee\ldots\vee \gamma_n)
  = \\
  & \sum_{k=0}^n\sum_{\sigma\in\mbox{\tiny Sh($k$,$n-k$)}}
  \epsilon(\sigma)Q_k(\gamma_{\sigma(1)}\vee\ldots\vee
  \gamma_{\sigma(k)})\vee\gamma_{\sigma(k+1)}\vee
  \ldots\vee\gamma_{\sigma(n)}.
\end{split}
\end{align} 
Here Sh($k$,$n-k$) denotes the set of $(k, n-k)$ shuffles in $S_n$,
$\epsilon(\sigma)=\epsilon(\sigma,\gamma_1,\ldots,\gamma_n)$ is a sign
given by the rule $
\gamma_{\sigma(1)}\vee\ldots\vee\gamma_{\sigma(n)}=
\epsilon(\sigma)\gamma_1\vee\ldots\vee\gamma_n $ and we use the
conventions that Sh($n$,$0$)=Sh($0$,$n$)$=\{\id\}$ and that the empty
product equals the unit. Note in particular that we also consider a
term $Q_0$ and thus we are actually considering curved
$L_\infty$-algebras.
Sometimes we also write $Q_k = Q_k^1$ and, following
\cite{canonaco:1999a}, we denote by $Q_n^i$ the component of $Q_n^i
\colon \Sym^n \mathfrak{L}[1] \rightarrow \Sym^i \mathfrak{L}[2]$ of $Q$. 
It is given by
\begin{align}
  \label{eq:Qniformula}
	\begin{split}
  & Q_n^i(x_1\vee \cdots \vee x_n)
	=  
	 \sum_{\sigma \in \mathrm{Sh}(n+1-i,i-1)}  \\
	& \epsilon(\sigma) Q_{n+1-i}^1(x_{\sigma(1)}\vee \cdots\vee  x_{\sigma(n+1-i)})\vee
	x_{\sigma(n+2-i)} \vee \cdots \vee x_{\sigma(n)},
\end{split}
\end{align}
where $Q_{n+1-i}^1$ are the usual structure maps. 
\begin{example}[Curved DGLA]
  \label{ex:curvedlie}
  A basic example of an $L_\infty$-algebra is that of a (curved) 
	differential graded Lie
  algebra $(\liealg{g},R,\D,[\argument,\argument])$ by setting
  $Q_0(1)={ -}R$, $Q_1={ -}\D$,
  $Q_2(\gamma\vee\mu)={ -(-1)^{|\gamma|}}[\gamma,\mu]$ and $Q_i=0$ for
  all $i\geq 3$.   Note that we denoted by $|\cdot |$ the degree in $\liealg{g}[1]$. 
\end{example}

Let us consider two $L_\infty$-algebras $(\mathfrak{L},Q)$ and
$(\widetilde{\mathfrak{L}},\widetilde{Q})$.  A degree $0$ counital
coalgebra morphism
\begin{equation*}
  F\colon 
  S^c(\mathfrak{L})
  \longrightarrow 
  S^c(\widetilde{\mathfrak{L}})
\end{equation*}
such that $FQ = \widetilde{Q}F$ is said to be an
\emph{$L_\infty$-morphism}.
A coalgebra morphism $F$ from $S^c(\mathfrak{L})$ to
$S^c(\widetilde{\mathfrak{L}})$ such that $F(1)=1$ is uniquely determined by its
components (also called \emph{Taylor coefficients})
\begin{equation*}
  F_n\colon \Sym^n(\mathfrak{L}[1])\longrightarrow \widetilde{\mathfrak{L}}[1],
\end{equation*}
where $n\geq 1$. Namely, we set $F(1)=1$ and use the formula
\begin{align}
\label{coalgebramorphism}
\begin{split}
  & F(\gamma_1\vee\ldots\vee\gamma_n)=
	\sum_{p\geq1}\sum_{\substack{k_1,\ldots, k_p\geq1\\k_1+\ldots+k_p=n}}
	\sum_{\sigma\in \mbox{\tiny Sh($k_1$,..., $k_p$)}}  \\
  &
  \frac{\epsilon(\sigma)}{p!}
  F_{k_1}(\gamma_{\sigma(1)}\vee\ldots\vee\gamma_{\sigma(k_1)})\vee\ldots\vee 
  F_{k_p}(\gamma_{\sigma(n-k_p+1)}\vee\ldots\vee\gamma_{\sigma(n)}),
\end{split}
\end{align}
where Sh($k_1$,...,$k_p$) denotes the set of $(k_1,\ldots,
k_p)$-shuffles in $S_n$ (again we set Sh($n$)$=\{\id\}$).
We also write $F_k = F_k^1$ and similarly to \eqref{eq:Qniformula} we get 
coefficients $F_n^j \colon \Sym^n \mathfrak{L}[1] \rightarrow 
\Sym^j \widetilde{\mathfrak{L}}[1]$ of $F$ by 
taking the corresponding terms in \cite[Equation~(2.15)]{dolgushev:2006a}. 
Note that $F_n^j$ only depends on $F_k^1 = F_k$ for $k\leq n-j+1$. 
Given an $L_\infty$-morphism $F$ of (non-curved) $L_\infty$-algebras $(\mathfrak{L},Q)$ and
$(\widetilde{\mathfrak{L}},\widetilde{Q})$,
we obtain the map of complexes
\begin{equation*}
  F_1\colon (\mathfrak{L},Q_1)\longrightarrow (\widetilde{\mathfrak{L}},\widetilde{Q}_1).
\end{equation*}
In this case the $L_\infty$-morphism $F$ is called an
\emph{$L_\infty$-quasi-isomorphism} if $F_1$ is a
quasi-isomorphism of complexes.
Given a DGLA $(\liealg{g}, \D, [\argument,\argument])$ and an element
$\pi\in \liealg{g}[1]^0$ we can obtain a curved Lie algebra by
defining a new differential $\D + [\pi,\argument]$ and considering the
curvature $R^\pi=\D\pi+\frac{1}{2}[\pi,\pi]$.
In fact the same procedure can be applied to a curved Lie algebra
$(\liealg{g}, R,\D, [\argument,\argument])$ to obtain the \emph{twisted}
curved Lie algebra $(\mathfrak{L}, R^\pi, \D +
[\pi,\argument],[\argument,\argument])$, where
\begin{equation*}
  R^\pi
  :=
  R+\D\pi + \frac{1}{2}[\pi,\pi].
\end{equation*} 
The element $\pi$ is called a \emph{Maurer--Cartan element} if it
satisfies the equation
\begin{equation}
\label{eq:DefMCEL}
  R+\D\pi+\frac{1}{2}[\pi,\pi]
  =
  0.
\end{equation}
Finally, it is important to recall that given a DGLA morphism, or more
generally an $L_\infty$-morphism, $F\colon \liealg{g}\rightarrow
\liealg{g}'$ between two DGLAs, one may associate to any 
(curved) Maurer--Cartan element $\pi\in\liealg{g}[1]^0$ 
a (curved) Maurer--Cartan element
\begin{equation}
\label{eq:FMC}
  \pi_F
	:=
	\sum_{n\geq 1} \frac{1}{n!}   F_n(\pi\vee\ldots\vee\pi)\in 
	\liealg{g}'[1]^0.
\end{equation}
In order to make sense of these infinite sums we consider DGLAs with complete 
descending filtrations
\begin{equation}
  \cdots 
	\supseteq 
	\mathcal{F}^{-2}\liealg{g}
	\supseteq 
	\mathcal{F}^{-1}\liealg{g}
	\supseteq 
	\mathcal{F}^{0}\liealg{g}
	\supseteq 
	\mathcal{F}^{1}\liealg{g}
	\supseteq 
	\cdots,
	\quad \quad
	\liealg{g}
	\cong
	\varprojlim \liealg{g}/\mathcal{F}^n\liealg{g}
\end{equation}  
and 
\begin{equation}
  \D(\mathcal{F}^k\liealg{g})
	\subseteq
	\mathcal{F}^k\liealg{g}
	\quad \quad \text{ and } \quad \quad
	[\mathcal{F}^k\liealg{g},\mathcal{F}^\ell\liealg{g}]
	\subseteq 
	\mathcal{F}^{k+\ell}\liealg{g}.
\end{equation}
In particular, $\mathcal{F}^1\liealg{g}$ is a projective limit of 
nilpotent DGLAs. In most cases the filtration is bounded below, i.e. 
bounded from the left with $\liealg{g}=\mathcal{F}^k\liealg{g}$ for some 
$k\in \mathbb{Z}$. If the filtration is unbounded, then we assume always that 
it is in addition exhaustive, i.e. that
\begin{equation}
  \liealg{g}
	=
	\bigcup_n \mathcal{F}^n\liealg{g},
\end{equation}
even if we do not mention it explicitly. Moreover, we assume that the 
DGLA morphisms are compatible with the filtrations.
Considering only Maurer--Cartan elements in $\mathcal{F}^1\liealg{g}^1$ ensures the 
well-definedness of \eqref{eq:FMC}. 
Mainly, the filtration is induced by formal power 
series in a formal parameter $\hbar$. Starting with a 
DGLA $(\liealg{g},\D,[\argument,\argument])$, its $\hbar$-linear extension to 
formal power series  $\liealg{G}=\liealg{g}[[\hbar]]$ of a 
DGLA $\liealg{g}$, has the complete descending filtration 
$\mathcal{F}^k \liealg{G} = \hbar^k \liealg{G}$. 

One can not only twist the DGLAs resp. $L_\infty$-algebras, but also the 
$L_\infty$-morphisms between them. Below we need the following result, see  
\cite[Prop.~2]{dolgushev:2006a} and \cite[Prop.~1]{dolgushev:2005b}.

\begin{proposition}
  \label{prop:twistinglinftymorphisms}
  Let $F\colon (\liealg{g},Q)\rightarrow (\liealg{g}',Q')$ be an 
	$L_\infty$-morphism of DGLAs, $\pi \in \mathcal{F}^1\liealg{g}^1$ 
	a Maurer-Cartan element and $S = F^1(\cc{\exp}(\pi))
	\in \mathcal{F}^1 \liealg{g}'^1$.
	\begin{propositionlist}					
		\item The map 
		      \begin{equation*}
					  F^\pi
						=
						\exp(-S\vee) F \exp(\pi\vee) \colon 
						\cc{\Sym}(\liealg{g}[1])
						\longrightarrow
						\cc{\Sym}(\liealg{g}'[1])
					\end{equation*}
					defines an $L_\infty$-morphism between the DGLAs 
					$(\liealg{g},\D+[\pi,\,\cdot\,])$ and $(\liealg{g}',\D+[S,\,\cdot\,])$.
					
		\item The structure maps of $F^\pi$ are given by 
		      \begin{equation}
					  \label{eq:twisteslinftymorphism}
					  F_n^\pi(x_1,\dots, x_n)
						=
						\sum_{k=0}^\infty \frac{1}{k!} 
						F_{n+k}(\pi, \dots, \pi,x_1 ,	\dots, x_n).
					\end{equation}
			
		\item Let $F$ be an $L_\infty$-quasi-isomorphism such that 
		      $F_1^1$ is not only a quasi-isomorphism of filtered complexes 
					$L\rightarrow L'$ but even induces a quasi-isomorphism
          \begin{equation*}
            F_1^1 \colon
	          \mathcal{F}^k L
	          \longrightarrow
	          \mathcal{F}^kL'
          \end{equation*}
          for each $k$. Then $F^\pi$ is an $L_\infty$-quasi-isomorphism.	 
	\end{propositionlist}
\end{proposition}

\subsection{Equivariant Polydifferential Operators}
\label{sec:EquivDpoly}

In the following we present some basic results concerning equivariant polydifferential
operators, which are basically folklore knowledge and are based on \cite{tsygan:note}.

Let us consider the DGLA of
\emph{polydifferential operators} on a smooth manifold $M$
\begin{equation}
  (\Dpoly^\bullet(M),\del= [\mu,\argument]_G,[\argument,\argument]_G)
\end{equation}
Here
\begin{equation*}
	  \Dpoly^\bullet(M)
		=
		\bigoplus_{n=-1}^\infty \Dpoly^{n}(M)
\end{equation*}
where $\Dpoly^{n}(M)= \Hom_{\mathrm{diff}}(\Cinfty(M)^{\otimes n+1},
\Cinfty(M))$ are the differentiable Hochschild cochains vanishing on constants. We use the sign convention from  \cite{bursztyn.dolgushev.waldmann:2012a} for the Gerstenhaber bracket 
$[\argument,\argument]$, not the original one from \cite{gerstenhaber:1963a}. Explicitly
\begin{equation}
  \label{eq:GerstenhaberBracketClassical}
  [D,E]_G
	=
	(-1)^{\abs{E}\abs{D}} \left(D \circ E - (-1)^{\abs{D} \abs{E}} E \circ D\right)
\end{equation}
with 
\begin{equation}
  D \circ E (a_0,\dots,a_{d+e}) 
  = 
	\sum_{i=0}^{\abs{ D}} (-1)^{i \abs{ E}} 
	D(a_0,\dots, a_{i-1}, E(a_i,\dots,a_{i+e}),a_{i+e+1},\dots,a_{d+e})
\end{equation}
for homogeneous $D,E \in \Dpoly^\bullet(M)$ and $a_0,\dots,a_{d+e}
\in \Cinfty(M)$. Moreover, $\mu$ denotes the commutative pointwise 
product on $\Cinfty(M)[[\hbar]]$ and $\del$ is the usual Hochschild 
differential.

We are interested in the case of group actions where we always consider a (left) action $\Phi\colon \group{G}\times M\to M$ of a connected Lie group $\group{G}$. Let $M$ be now equipped with a 
$\group{G}$-invariant star product 
$\star$, i.e. an associative product $\star = \mu + \sum_{r=1}^\infty 
\hbar^r C_r= \mu_0 +\hbar m_\star \in (\Dpoly^{1}(M))^\group{G}[[\hbar]]$. Recall that a linear map $H\colon \liealg{g} \to \Cinfty(M)[[\hbar]]$ is called a
\emph{quantum momentum map} if
\begin{align*}
  \Lie_{\xi_M}
  =
  -\frac{1}{\hbar}[H(\xi),\argument]_\star 
  \ \text{ and }\ 
  \frac{1}{\hbar}[H(\xi),H({\eta})]_\star
  =
  H([\xi,\eta]),
\end{align*}
where $\xi_M$ denotes the fundamental vector field corresponding to
the action $\Phi$.

A pair $(\star,H)$ consisting of an invariant star product 
$\star = \mu + \hbar m_\star$ and a quantum momentum map $H$ is also called \emph{equivariant star product}. They are useful since they 
allow for a BRST like reduction scheme, compare 
Appendix~\ref{sec:BRSTReductionStarProducts}. We introduce now 
the DGLA that contains the data of Hamiltonian actions, 
i.e. of equivariant star products. Here we follow \cite{tsygan:note}.
\begin{definition}[Equivariant polydifferential operators]
  The DGLA of \emph{equivariant polydifferential operators} 
	$(D^\bullet_{\liealg{g}}(M),\del^\liealg{g},[\argument,\argument]_\liealg{g})$ 
	is defined by
  \begin{align}
    D^k_{\liealg{g}}(M)
    = 
    \bigoplus_{2i+j=k} (\Sym^i\liealg{g}^* \tensor D_{\mathrm{poly}}^j(M))^\group{G}
  \end{align}
  with bracket 
  \begin{align}
    [\alpha \otimes D_1,\beta \otimes D_2]_\liealg{g}
    =
    \alpha \vee \beta \otimes [D_1,D_2]_G
  \end{align} 
  and differential 
  \begin{align}
    \partial^\liealg{g} (\alpha \otimes D_1)
    =
		\alpha \otimes \del D_1
		=
    \alpha \otimes [\mu,D_1]_G
  \end{align}
	for $\alpha\otimes D_1,\beta \otimes D_2 \in D^\bullet_\liealg{g}(M)$. 
  Here we denote by $\partial $ and $[\argument,\argument]_G$ the usual Hochschild 
	differential and Gerstenhaber bracket on the polydifferential operators 
	and by $\mu$ the pointwise multiplication of $\Cinfty(M)$.
\end{definition}

Notice that invariance with respect to the group action means
invariance under the transformations $\Ad_g^*\tensor \Phi_g^*$ for all
$g\in G$, and that the equivariant polydifferential operators can be interpreted as 
equivariant polynomial maps $\liealg{g} \to  D_{\mathrm{poly}}(M)$. 
We introduce the canonical linear map
\begin{align*}
  \lambda
  \colon
  \liealg{g} \ni \xi \longmapsto \Lie_{\xi_M}\in D^0_{\mathrm{poly}}(M),
\end{align*}	
and see that $\lambda \in D^2_{\liealg{g}}(M)$ is central and moreover
$\partial^\liealg{g} \lambda = 0$. This implies that we can see $D^\bullet_\liealg{g}(M)$
either as a flat DGLA with the above structures or as a curved DGLA with 
the above structures and curvature $\lambda$.
In the case of formal power series we rescale 
the curvature again by $\hbar^2$ and obtain the following 
characterization of Maurer-Cartan elements:
\begin{lemma}\label{Lem: MCtoStar} 
  A curved formal Maurer-Cartan element $\Pi\in \hbar D^1_{\liealg{g}}(M)[[\hbar]]$, 
	i.e. an element $\Pi$ satisfying
  \begin{align}
    \label{eq:MCquantum}
    \hbar^2\lambda + \partial^\liealg{g} \Pi + \frac{1}{2}[\Pi,\Pi]_\liealg{g}
    =
    0,
  \end{align}
  is equivalent to a pair $(m_\star,H)$, where $m_\star 
	\in D_{\mathrm{poly}}^1 (M))^\group{G}[[\hbar]]$ defines a 
	$\group{G}$-invariant star product via $\star = \mu+\hbar m_\star$ 
	with quantum momentum map 
	$H\colon \liealg{g} \to \Cinfty(M)[[\hbar]]$. In other words, $(\star, H)$ is an 
	equivariant star product.
\end{lemma}
\begin{proof}
  Let us decompose $\Pi=\hbar m_\star - \hbar H \in \hbar (
  D_{\mathrm{poly}}^1 (M))^\group{G} \oplus (\liealg{g}^* \tensor
  D_{\mathrm{poly}}^{-1}(M))^\group{G}[[\hbar ]]$. Then the curved 
	Maurer-Cartan equation applied to an element $\xi \in \liealg{g}$ reads 
  	\begin{align*}
  	-\hbar^2 \Lie_{\xi_M}
		&=-\hbar^2\lambda(\xi)
  	=
  	\partial^\liealg{g} \Pi(\xi) + \frac{1}{2}[\Pi,\Pi]_\liealg{g}(\xi)\\&
  	=
  	\hbar [\mu,m_\star]_G+ \frac{\hbar^2}{2}[m_\star,m_\star]_G- \hbar^2[m_\star, H(\xi)]_G.
	\end{align*}
This is equivalent to the fact that $\hbar m_\star$ is Maurer-Cartan in the flat setting 
and that  $\Lie_{\xi_M}=-\frac{1}{\hbar}[H(\xi),-]_\star$, since $\hbar[m_\star, H(\xi)]_G(f)
=-[H(\xi),f]_{\star}$ for $f\in \Cinfty(M)$. Then the invariance of both elements implies 
that $\star=\mu+\hbar m_\star$ is a $\group{G}$-invariant star product with quantum 
momentum map $H$. 
\end{proof}
Two equivariant star products $\hbar (m_\star-H)$ and $\hbar(m_\star' - H')$ are 
called \emph{equivariantly equivalent} if they are gauge equivalent, i.e. if 
there exists an $\hbar T\in \hbar D_\poly^0(M)^{\group{G}}[[\hbar]] \subset 
D^0_\liealg{g}(M)$ such that
\begin{equation*}
  \hbar(m_\star' - H')
	=
	\exp(\hbar[T,\argument]_{\liealg{g}}) \acts \hbar (m_\star-H) 
	=
  \exp(\hbar[T,\argument]_{\liealg{g}})(\mu+ \hbar (m_\star-H)) - \mu.
\end{equation*}
This means that $S = \exp(\hbar T)$ satisfies for all $f,g\in\Cinfty(M)[[\hbar]]$
\begin{equation*}
  S(f\star g)
	=
	Sf \star' Sg
	\quad\text{ and }\quad
	SH
	=
	H'.
\end{equation*}

%
%
\section{Reduction of the Equivariant Polydifferential Operators}
\label{sec:RedDpoly}

Now we aim to describe a reduction scheme for general equivariant polydifferential operators via an $L_\infty$-morphism denoted by $\mathrm{D}_\red$,  generalizing the results for the 
polyvector fields from \cite{esposito.kraft.schnitzer:2020a:pre}.

Let $M$ be a smooth manifold with 
action $\Phi \colon \group{G} \times M \rightarrow M$ of a connected Lie group 
and let $(\star, H= J + \hbar J')$ be an equivariant star product, i.e. a curved 
formal Maurer-Cartan element in the equivariant polydifferential operators, 
see Lemma~\ref{Lem: MCtoStar}. 
Here the component $J\colon M \rightarrow \liealg{g}^*$ of the quantum momentum map $H$ 
in $\hbar$-order zero is a classical momentum map with respect to the Poisson structure 
induced by the skew-symmetrization of the $\hbar^1$-part of $\star$. 
We assume from now on that $0 \in \liealg{g}^*$ is a value and a regular value of 
$J$ and set $C= J^{-1}(\{0\})$. In addition, we require the action to be proper
 around $C$ and free on $C$. 
Then $M_\red = C/\group{G}$ is a smooth manifold and we denote 
by $\iota \colon C \rightarrow M$ the inclusion and by $\pr \colon C 
\rightarrow M_\red$ the projection on the quotient.
Moreover, the properness around $C$ implies that there exists an
  $\group{G}$-invariant open neighbourhood $M_\nice \subseteq M$ of $C$ and a 
  $\group{G}$-equivariant diffeomorphism $\Psi \colon M_\nice
  \rightarrow U_\nice \subseteq C\times \liealg{g}^*$, where $U_\nice$ is an open 
	neighbourhood of $C\times\{0\}$ in $C\times \liealg{g}^*$. Here the Lie group $\group{G}$ 
	acts on $C \times \liealg{g}^*$ as $\Phi_g= \Phi^C_g\times \Ad^*_{g^{-1}}$, 
	where $\Phi^C$ is the induced action on $C$, and the momentum map on $U_\nice$ is 
	the projection to $\liealg{g}^*$ (see \cite[Lemma~3]{bordemann.herbig.waldmann:2000a}, \cite{gutt.waldmann:2010a})

From now on we assume $M=M_\nice$. Then we can 
define an equivariant \emph{prolongation map}  
by 
\begin{equation*}
  \prol \colon 
	\Cinfty(C) \ni 
	\phi
	\longmapsto
	(\pr_1 \circ \Psi)^* \phi \in
	\Cinfty(M_\nice)
\end{equation*}
and we directly get $\iota^*\prol = \id_{\Cinfty(C)}$.

Consider the Taylor expansion around $C$ 
in $\liealg{g}^*$-direction
as in 
\cite[Section 4.1]{esposito.kraft.schnitzer:2020a:pre},
 which is a map
\begin{align*}
  D_{\liealg{g}^*}\colon
  \Dpoly^k(C\times \liealg{g}^*)
  \longmapsto
  \prod_{i=0}^\infty(\Sym^i\liealg{g}\tensor T^{k+1}(\Sym \liealg{g}^*)\tensor\Dpoly^k(C)),
\end{align*}
where $T^\bullet(\Sym \liealg{g}^*)$ denotes the tensor module of 
$\Sym\liealg{g}^*$. Note that we are only interested in a 
subspace since we consider polydifferential operators vanishing 
on constants. Slightly abusing the notation, the Taylor expansion of the equivariant polydifferential operators takes then the following form 
\begin{equation}
	\label{eq:TaylorDpoly}
	D_\Tay(C\times \liealg{g}^*)
	=
	\left( \Sym \liealg{g}^* \otimes 
	\prod_{i=0}^\infty (\Sym^i \liealg{g} \otimes 
	T(\Sym \liealg{g}^*) 	\otimes \Dpoly(C))
	\right)^\group{G}
\end{equation}
and one easily checks that this yields an equivariant DGLA morphism 
\begin{align}
  D_{\liealg{g}^*}\colon 
  (D_\liealg{g}(M),\lambda,\del^\liealg{g},[\argument,\argument]_\liealg{g}) 
  \longrightarrow 
  (D_\Tay(C\times \liealg{g}^*),\lambda,\del,	[\argument,\argument]).
\end{align}
Our goal consists in finding a reduction morphism from
\begin{equation}
	\label{eq:curvedTaylorDpoly}
	D_\red \colon
  \left(D_\Tay(C\times \liealg{g}^*)[[\hbar]], \hbar\lambda,
	\del +[-J,\argument],[\argument,\argument] \right)
  \longrightarrow
	(D_\poly(M_\red)[[\hbar]],\del,[\argument,\argument]_G).
\end{equation}
Following a similar strategy as in \cite{esposito.kraft.schnitzer:2020a:pre}, 
we construct $L_\infty$-morphisms
\begin{equation*}
  D_\Tay(C\times \liealg{g}^*)[[\hbar]] 
	\longrightarrow 
	\left( \prod_{i=0}^\infty (\Sym^i \liealg{g} \otimes 
	\Dpoly(C))\right)^\group{G}[[\hbar]]
	\longrightarrow
	D_\poly(M_\red)[[\hbar]]
\end{equation*}
with suitable $L_\infty$-structures on the three spaces, 
where $\left( \prod_{i=0}^\infty (\Sym^i \liealg{g} \otimes \Dpoly(C))\right)^\group{G}[[\hbar]]$ is a candidate for a 
Cartan model.

%
%
%

%
%
%
\subsection{A 'partial' Homotopy for the Hochschild Differential}

In order to find a suitable analogue of the Cartan model for the polydifferential 
operators, we need to understand the cohomology of 
\begin{equation*}
  (D_{\liealg{g}}(M), \del^{\liealg{g}} - [J,\argument]_{\liealg{g}}, 
	[\argument,\argument]_{\liealg{g}})
\end{equation*}
and in particular the role of the differential $[-J,\argument]_{\liealg{g}}$. 
To this end we 
construct a 'partial' homotopy for $\del^{\liealg{g}} -[J,\argument]_{\liealg{g}}$. 
Here we use the results concerning
the homotopy 
for the Hochschild differential from \cite{dewilde.lecomte:1995a}. 
In particular, we restrict ourselves to the subspace of 
normalized differential Hochschild cochains, i.e. polydifferential 
operators vanishing on constants. One can show that they are 
quasi-isomorphic to the differential ones.
Recall the maps
\begin{align}
  \begin{split}
  \label{eq:PhiGeneral}
	\Phi \colon \Dpoly^a(M) 
	& \longrightarrow 
	\Dpoly^{a-1}(M) \\
	\Phi(A)(f_0,\dots,f_{a-1})
	& =
	\sum_{t=1}^n\sum_{i\leq j<a}(-1)^i A\left(f_0,\dots,f_{i-1},x^t,\dots,
	\frac{\del}{\del x^t} f_j,\dots, f_{a-1}\right)
	\end{split}
\end{align}
for $f_1,\dots,f_{a-1} \in \Cinfty(M)$ and
\begin{equation}
  \label{eq:PsiGeneral}
	\Psi \colon \Dpoly^a(M) \ni A
	\longmapsto
	(-1)^{a} [x^i,A]_G \cup \frac{\del}{\del x^i}
	=
	(-1)^{a+1} \sum_{i=1}^n (A \circ x^i) \cup \frac{\del}{\del x^i} \in 
	\Dpoly^a(M).
\end{equation}
for local coordinates $(x^1,\dots,x^n)$ of $M$. They satisfy, by 
\cite[Proposition~4.1]{dewilde.lecomte:1995a}, the following condition:
\begin{equation*}
  \Phi \circ \del + \del \circ \Phi 
	=
	-(\deg_D \cdot \id + \Psi),
\end{equation*}
where $\deg_D$ is the order of the differential operator. 

We assume from now on for simplicity $M= C \times \liealg{g}^*$ and 
$J= \pr_{\liealg{g}^*}$ and we want 
to find a suitable Cartan model for the polydifferential operators. 
Similarly to \cite[Definition~4.14]{esposito.kraft.schnitzer:2020a:pre} 
for the polyvector field case, 
we want to obtain a DGLA structure on
\begin{equation*}
  \left( \prod_{i=0}^\infty( \Sym^i \liealg{g} \otimes \Dpoly(C))\right)^\group{G}.
\end{equation*}
Hence we adapt the maps $\Phi$ and $\Psi$ in such a way that they only 
include coordinates $J_i = \alpha_i = e_i$ on $\liealg{g}^*$ with $i=1,\dots,n$:
\begin{align}
\label{eq:DefPhiiPsi}
\begin{split}
  \Phi(A)(f_0,\dots,f_{a-1})
	& =
	\sum_{t=1}^n\sum_{i \leq j<a}(-1)^i A\left(f_0,\dots,f_{i-1},e_t,\dots,
	\frac{\del}{\del e_t} f_j,\dots, f_{a-1}\right), \\
	\Psi(A)
	& =
	(-1)^{a+1} \sum_{i=1}^n (A \circ e_i) \cup \frac{\del}{\del e_i},
\end{split}
\end{align}
where $A\in \Dpoly^a(C\times\liealg{g}^*) $ and 
$f_0,\dots, f_{a-i} \in \Cinfty(C\times\liealg{g}^*)$. 

\begin{proposition}
  One has on $\Dpoly(C\times\liealg{g}^*)$
	\begin{equation}
	  \label{eq:Phidel}
	  \Phi \circ \del + \del \circ \Phi
		=
		- (\deg_\liealg{g}\cdot \id +\Psi),
	\end{equation}
	where $\deg_\liealg{g}$ is the order of differentiations in direction 
	of $\liealg{g}^*$-coordinates.
\end{proposition}
\begin{proof}
The proof follows the same lines as in \cite[Proposition~4.1]{dewilde.lecomte:1995a}. 
It is proven by induction on the degree of $a$ of $A\in \Dpoly^a(C\times\liealg{g}^*)$. 
For $a=0$ and $A \in D_\poly^0(C\times\liealg{g}^*)$ as well as 
$f \in \Cinfty(C\times\liealg{g}^*)$ we get
\begin{align*}
  ((\Phi \circ \del + \del \circ \Phi)(A))(f)
	& =
	(\del A)\left(e_i, \frac{\del}{\del e_i}f\right)
	=
	e_i A\left(\frac{\del}{\del e_i}f\right) - A\left(e_i\frac{\del}{\del e_i}f\right)
	+ A(e_i)\frac{\del}{\del e_i}f  \\
	& = 
	(-\deg_\liealg{g}(A)\, A - \Psi(A))(f).
\end{align*}
Note that $\Psi$ has the following compatibility with the $\cup$-product:
\begin{equation*}
  \Psi(A\cup B)
	=
	(\Psi A)\cup B + A \cup (\Psi B)
	+
	(-1)^a (A \circ e_i)\cup \left( \frac{\del}{\del e_i}\cup B + 
	(-1)^b B \cup \frac{\del}{\del e_i}\right).
\end{equation*}
Writing $\ins(A)(\argument)= (\argument)\circ A$ one computes
\begin{align*}
  (\Phi \circ \del + \del \circ \Phi)(A\cup B)
	& = 
	((\Phi \circ \del + \del \circ \Phi)A)\cup B + 
	A \cup ((\Phi \circ \del + \del \circ \Phi)B) \\
	& + 
	((\ins(e_i)\circ \del + \del \circ \ins(e_i))A) \cup 
	\ins\left( \frac{\del}{\del e_i} \right)B  \\
	& +
	(-1)^a (\ins(e_i)A)\cup (\del\circ \ins\left(\frac{\del}{\del e_i}\right) - 
	\ins\left(\frac{\del}{\del e_i}\right)\circ \del)B.
\end{align*}
The operators $(\ins(e_i)\circ \del + \del \circ \ins(e_i))$ 
and $(\del\circ \ins\left(\frac{\del}{\del e_i}\right) - 
\ins\left(\frac{\del}{\del e_i}\right)\circ \del)$ are graded commutators of 
derivations of the $\cup$-product and therefore themselfes graded derivations. 
Thus they are determined by their action on $\Dpoly^{-1}(C\times\liealg{g}^*)$ 
and $\Dpoly^0(C\times\liealg{g}^*)$. The first one obviously vanishes. 
The second coincides on these generators with 
\begin{equation*}
  A
	\longmapsto
	- \left( \frac{\del}{\del e_i}\cup A +(-1)^a A\cup \frac{\del}{\del e_i}\right)
\end{equation*}
and the proposition is shown.
\end{proof}
As in \cite{esposito.kraft.schnitzer:2020a:pre}, we
define a homotopy on the equivariant 
polydifferential operators
\begin{equation*}
  \widehat{h} \colon
	(\Sym\liealg{g}^* \otimes\Dpoly^d(C\times\liealg{g}^*))^\group{G} \ni P \otimes D
	\longmapsto 
	(-1)^{d+1} \inss(e_i)P \otimes D \cup \frac{\del}{\del e_i}
	\in (\Sym\liealg{g}^* \otimes\Dpoly^{d+1}(C\times\liealg{g}^*))^\group{G} .
\end{equation*}
The fact that $\widehat{h} $ maps invariant elements to invariant ones follows as 
in the case of polyvector fields. Finally, note that $\Phi$ and $\Psi$ are equivariant, 
whence they can be extended to the equivariant polydifferential operators, where we can show:

\begin{proposition}
  One has on $(\Sym\liealg{g}^* \otimes\Dpoly(C\times\liealg{g}^*))^\group{G}$
	\begin{equation}
	  \left[\widehat{h} -\Phi, \del^{\liealg{g}} + [-J,\argument]_{\liealg{g}}\right]
		=
		(\deg_{\Sym\liealg{g}^*} + \deg_\liealg{g}) \id,
	\end{equation}
	where $\deg_\liealg{g}$ is again the order of differentiations in direction 
	of $\liealg{g}^*$-coordinates.
\end{proposition}
\begin{proof}
From \eqref{eq:Phidel} we know $[\Phi,\del^{\liealg{g}}] = - (\deg_\liealg{g}\cdot \id +\Psi)$. 
In addition, one has for homogeneous $P\otimes D$
\begin{align*}
  \widehat{h}  \circ \del^{\liealg{g}} (P\otimes D)
	& =
	(-1)^{d +2} \inss(e_i)P \otimes ( \del D) \cup \frac{\del}{\del e_i} 
	=
	-(-1)^{d +1} \inss(e_i)P \otimes \del (D \cup \frac{\del}{\del e_i}) \\
	& =
	- \del^{\liealg{g}}\circ \widehat{h}  (P\otimes D).
\end{align*}
Moreover, since we consider only differential operators vanishing on constants 
one checks easily that also $ [\Phi, [-J,\argument]_{\liealg{g}}] 	=	0.$ Finally,
\begin{align*}
  [\widehat{h} ,[-J,\argument]_{\liealg{g}}] (P\otimes D)
	& =
	(-1)^d \inss(e_i)(e^j \vee P)\otimes 
	[- J_j, D] \cup \frac{\del}{\del e_i}  \\
	& \quad +
	(-1)^{d+1} e^j \vee \inss(e_i)P \otimes 
	\left[-J_j, D \cup \frac{\del}{\del e_i} \right]  \\
	& =
	- \Psi(P\otimes D) + (-1)^d e^j\vee \inss(e_i)P\otimes 
	[-J_j,D] \cup \frac{\del}{\del e_i} \\
	& \quad+
	(-1)^{d+1} e^j\vee \inss(e_i)P\otimes 
	[-J_j,D] \cup \frac{\del}{\del e_i}
	+\deg_{\Sym\liealg{g}^*}(P) \,P \otimes D \\
	& =
	(\deg_{\Sym\liealg{g}^*}\cdot \id -\Psi) P\otimes D .
\end{align*}
Thus the proposition is shown.
\end{proof}

The above contructions work also for the Taylor series expansion of the equivariant polydifferential operators, where we restrict ourselves again 
to polydifferential operators vanishing on constants. We slightly abuse the notation and 
denote them again by $D_\Tay(C\times\liealg{g}^*)$, compare 
\eqref{eq:TaylorDpoly}. Writing
\begin{align}
  \label{eq:totalhomotopy}
  h
	=
  \begin{cases}
  & \frac{1}{\deg_{\Sym\liealg{g}^*} + \deg_\liealg{g}} (\widehat{h}  - \Phi)
	\quad \quad \text{ if }
	\deg_{\Sym\liealg{g}^*}+ \deg_\liealg{g} \neq 0, \\
	&0  \quad \quad\quad \;\quad \; \hspace{2cm}\text{ else,}
	\end{cases}
\end{align}
we get the following result:

\begin{proposition}
  One has a deformation retract  
	\begin{equation}
	\label{eq:ClassicalRetractDTay}
	  \begin{tikzcd} 
      \left(\left( \prod_{i=0}^\infty (\Sym^i \liealg{g} \otimes 
			\Dpoly(C))\right)^\group{G}[[\hbar]],\del\right) 
			\arrow[rr, shift left, "i"] 
      &&  
			(D_\Tay(C\times \liealg{g}^*)[[\hbar]], \del + [-J,\argument]), 
			\arrow[loop right, 
	distance=3em, start anchor={[yshift=1ex]east}, end anchor={[yshift=-1ex]east}]{}{h}
			\arrow[ll, shift left, "p"]
	  \end{tikzcd}
  \end{equation}
	where $p$ and $i$ denote the obvious projection and inclusion. This means that one has 
	$pi=\id$ and $\id - ip = [h, \del + [-J,\argument]]$. 
	Moreover, the identities $hi=0 = ph$ hold.
\end{proposition}

\begin{remark}
  Note that one has $h^2\neq 0$, i.e. the above retract is no special 
	deformation retract. However, by the results of 
	\cite[Remark~2.1]{huebschmann:2011a} we know that 
	this could also be achieved.
\end{remark}

The reduction works now in two steps: At first, we use the 
homological perturbation lemma from Lemma~\ref{lemma:homologicalperturbationlemma}
to deform the differential on 
$D_\Tay(C\times\liealg{g}^*)[[\hbar]]$, and in a second step we use the 
homotopy transfer theorem, see Theorem~\ref{thm:HTTJonas}, 
to extend the deformed projection 
to an $L_\infty$-morphism. This will possibly give us higher brackets on 
$\left( \prod_{i=0}^\infty (\Sym^i \liealg{g} \otimes 
			\Dpoly(C))\right)^\group{G}[[\hbar]]$ that 
we have to discuss.

\subsection{Application of the Homological Perturbation Lemma}

In our setting, the bundle $C \times \liealg{g} \rightarrow C$ can be equipped with the 
structure of a Lie algebroid since $\liealg{g}$ acts on $C$ by the 
fundamental vector fields. The bracket of this \emph{action Lie algebroid} 
is given by 
\begin{equation}
	\label{eq:BracketActionAlg}
	[\xi, \eta]_{C\times\liealg{g}}(p) 
	=
	[\xi(p),\eta(p)] - (\Lie_{\xi_C}\eta)(p) + 
	(\Lie_{\eta_C}\xi)(p)
\end{equation}
for $\xi,\eta \in \Cinfty(C,\liealg{g})$. The anchor is given by 
$\rho(p,\xi) = -\xi_C\at{p}$. In particular, one can check that 
$\pi_\KKS$ is the negative of the linear Poisson structure on its dual 
$C \times \liealg{g}^*$ in the convention of \cite{neumaier.waldmann:2009a}.

For Lie algebroids there is a well-known construction of universal 
enveloping algebras \cite{moerdijk.mrcun:2010a,neumaier.waldmann:2009a,rinehart:1963a}. It turns out that in our special case we get a 
simpler description of the universal enveloping algebra:

\begin{proposition}
\label{prop:poductonUniversalEnvelopingAlg}
  The universal enveloping algebra $\Universal(C\times \liealg{g})$ of 
	the action Lie algebroid $C\times \liealg{g}$ is isomorphic to 
	$\Cinfty(C) \rtimes \Universal(\liealg{g})$ with product 
  \begin{equation}
    (f,x) \cdot (g,y)
	  =
	  \sum \left(f \Lie(x_{(1)})(g) ,x_{(2)}y\right) .
  \end{equation}
	Here $ y_{(1)} \otimes y_{(2)}= \Delta(y)$ denotes the coproduct on 
	$\Universal(\liealg{g})$ induced by extending 
	$\Delta(\xi) = 1 \otimes \xi + \xi \otimes 1$ as an algebra morphism. 
	Moreover, 
	$\Lie \colon \Universal(\liealg{g}) \rightarrow \Diffop(\Cinfty(C))$ is 
	the extension of the anchor of the action algebroid, i.e. of the 
	negative fundamental vector fields, to the universal enveloping 
	algebra. The same holds also in the formal setting of 
	$\Universal_\hbar(\liealg{g})$ with bracket rescaled 
	by $\hbar$. Note that in this case one has to rescale $\Lie$ by 
	powers of $\hbar$, i.e. 
	$\Lie_\xi = -\hbar \Lie_{\xi_C}$ for $\xi \in \liealg{g}$.
\end{proposition}
\begin{proof}
Note that the product is associative since 
\begin{align*}
  \left((f,x)\cdot (g,y)\right) \cdot (h,z)
	& = 
	\sum (f \Lie(x_{(1)})g,x_{(2)}y)\cdot (h,z)   
	=
	\sum (f \Lie(x_{(1)})g \Lie(x_{(2)}y_{(1)})h, x_{(3)}y_{(2)}z) \\
	& =
  \sum (f,x) \cdot (g \Lie(y_{(1)})h, y_{(2)}z) 
	=
	(f,x)\cdot \left((g,y)\cdot (h,z) \right), 
\end{align*}
where the penultimate identity follows with the coassociativity of $\Delta$ 
and the identity $\Lie(x)(fg)= \Lie(x_{(1)})(f) \Lie(x_{(2)})(g)$. Note that the inclusions 
$\kappa_C \colon \Cinfty(C) \rightarrow \Cinfty(C) \rtimes 
\Universal(\liealg{g})$ and $\kappa \colon \Cinfty(C)\otimes\liealg{g} 
\rightarrow \Cinfty(C) \rtimes \Universal(\liealg{g})$ satisfy
\begin{equation*}
  [\kappa(s),\kappa_C(f)]
  =
  \kappa(\rho(s)f), 
  \quad \quad \text{ and } \quad \quad
  \kappa_C(f)\kappa(s)
  =
  \kappa(fs).
\end{equation*}
Thus the universal property gives the desired morphism 
$\Universal(C\times\liealg{g}) \rightarrow 
\Cinfty(C) \rtimes \Universal(\liealg{g})$. 
Recursively we can show that the right hand side is generated 
by $u \in \Cinfty(C)$ and $\xi \in \Cinfty(C)\otimes \liealg{g}$ 
which gives the surjectivity of the morphism. 
Concerning injectivity, suppose $(f^{i_1} ,e_{i_1})\cdots
(f^{i_n}, e_{i_n})= 0$ in 
$\Cinfty(C) \rtimes \Universal(\liealg{g})$. 
We have to show that also 
$(f^{i_1} e_{i_1})\cdots (f^{i_1} e_{i_1}) = 0 $ in 
$\Universal(C\times\liealg{g})$. But this follows from a direct 
comparison of the terms in the corresponding associated 
graded algebras.
\end{proof}

Recall that by the Poincar\'e-Birkhoff-Witt Theorem the map
\begin{equation*}
  \Sym(\liealg{g}) \ni 
	x_1 \vee \cdots \vee x_n 
	\longmapsto 
	\frac{1}{n!} \sum_{\sigma\in S_n} x_{\sigma(1)} \cdots x_{\sigma(n)} \in 
	\Universal(\liealg{g})
\end{equation*}
is a coalgebra isomorphism with respect to the usual coalgebra structures induced by 
extending $\Delta(\xi)=\xi \otimes 1 +1 \otimes \xi $ for $\xi \in \liealg{g}$, 
see e.g. \cite{berezin:1967a,higgins:1969a}. 
This statement holds also in the case of formal power series in $\hbar$ whence we can 
transfer the product on the universal enveloping algebra 
as in Proposition~\ref{prop:poductonUniversalEnvelopingAlg} to an associative 
product $\star_G = \mu + \hbar m_G$ on $\Cinfty(C)\otimes \Sym (\liealg{g})[[\hbar]]$.

\begin{lemma}
\label{lem:GuttProduct}
  The \emph{Gutt product} $\star_G$ on $\Cinfty(C)\otimes \Sym (\liealg{g})[[\hbar]]$ is 
	$\group{G}$-invariant and $J = \pr_{\liealg{g}^*} \colon M= C\times \liealg{g}^* \rightarrow 
	\liealg{g}^*$ is a momentum map, i.e.
	\begin{equation}
	  - \Lie_{\xi_M}
		=
		\frac{1}{\hbar} \ad_{\star_G} (J(\xi)).
	\end{equation}
\end{lemma}
\begin{proof}
Both statements follow directly from the explicit formula in 
Proposition~\ref{prop:poductonUniversalEnvelopingAlg}.
\end{proof}

We deform the differential 
$\del + [-J,\argument]$ by $[\hbar m_G,\argument]$, i.e. exactly by the 
higher orders of this product. The perturbed differential $\del^{\liealg{g}} 
+ [\hbar m_G-J,\argument]= [\star_G -J,\argument]$ squares indeed to zero since 
we have with the above lemma
\begin{equation*}
  [\star_G - J, \argument]^2
	=
	\frac{1}{2} [ [\star_G -J,\star_G-J],\argument]
	=
	[-\hbar \lambda ,\argument]
	=
	0,
\end{equation*}
where again $\lambda = e^i \otimes (e_i)_M$. By the homological 
perturbation lemma as formulated in Section~\ref{sec:HomPerLemma} 
this yields a homotopy retract
\begin{equation}
  \label{eq:defomedRetractDtay}
	  \begin{tikzcd} 
      \left(\left( \prod_{i=0}^\infty (\Sym^i \liealg{g} \otimes 
			\Dpoly(C))\right)^\group{G}[[\hbar]],\del_\hbar\right) 
			\arrow[rr, shift left, "i_\hbar"] 
      &&  
			(D_\Tay(C\times \liealg{g}^*)[[\hbar]], [\star_G-J,\argument]),
      \arrow[ll, shift left, "p_\hbar"]
	    \arrow[loop right, 
	distance=3em, start anchor={[yshift=1ex]east}, end anchor={[yshift=-1ex]east}]{}{h_\hbar}
	  \end{tikzcd}
  \end{equation}	
with $B = [\hbar m_G,\argument]$ and
\begin{equation}
    \begin{alignedat}{5}
      \label{eq:deformedhomologicalmapsDTay}
      A 
      &= 
      (\id + Bh)^{-1}B, \quad \quad
      &&
      \del_\hbar
      &&= 
      \del + p Ai  , \quad \quad
			&&
      i_\hbar
      &&= 
      i -hAi ,     \\ 
      p_\hbar 
      &=  
      p-pAh  ,
      && 
      h_\hbar 
      &&= 
      h-hAh,
    \end{alignedat}
  \end{equation}
compare Lemma~\ref{lemma:homologicalperturbationlemma}.
More explicitly, we have
\begin{equation}
\label{eq:DeformedIandH}
	i_\hbar
	=
	\sum_{k=0}^\infty  (\widetilde{\Phi} \circ B)^k \circ i, 
	\quad \text{ and }\quad
	h_\hbar
	=
	h \circ \sum_{k=0}^\infty (-Bh)^k,
\end{equation}
where $\widetilde{\Phi}$ is the the combination of $\Phi$ with the 
degree-counting coefficient from $h$ from \eqref{eq:totalhomotopy}. 
We want to take a closer look at the induced differential:

\begin{proposition}
\label{prop:DtayHomPerDiff}
  One has
  \begin{equation}
	  p_\hbar = p
		\quad \text{ and }\quad 
		\del_\hbar
		=
		\del + \delta
	\end{equation}
	with 
	\begin{align*}
	  \delta (P\otimes D)
		=
		(-1)^d P_{(1)} \otimes D \cup \Lie_{P_{(2)}} 
		- (-1)^d P \otimes D \cup \id
	\end{align*}
	for homogeneous $P \otimes D \in \Sym\liealg{g} \otimes \Dpoly^d(C)$.
\end{proposition}
\begin{proof}
The fact that $p_\hbar = p$ follows since $Bh$ always adds differentials 
in $\liealg{g}$-direction. For the deformed differential we compute 
for homogeneous $P \otimes D \in \Sym\liealg{g} \otimes \Dpoly^d(C)$ 
and $f_i\in \Cinfty(C)$
\begin{align*}
  (\delta (P\otimes D)) (f_0,f_1,\dots,f_{d+1})
	& =
	(p \circ \sum_{k=0}^\infty (B\circ\widetilde{\Phi} )^k B \circ i (P\otimes D))
	(f_0,f_1,\dots,f_{d+1}) \\
	& =
	p  ( B(P\otimes D)(f_0,f_1,\dots,f_{d+1}))\\
	& =
	(-1)^d p ( \hbar m_G( P \otimes D(f_0,\dots,f_d), f_{d+1}) \\
	& =
	(-1)^d P_{(1)}\otimes D(f_0,\dots,f_d) \cdot \Lie_{P_{(2)}}f_{d+1}
\end{align*}
for all $P_{(2)}\neq 1$.
Here we used the explicit form of the Gutt product as in 
Proposition~\ref{prop:poductonUniversalEnvelopingAlg} and the fact that 
$\Sym(\liealg{g})[[\hbar]]$ and $\Universal_\hbar(\liealg{g})$ are isomorphic coalgebras.
\end{proof}

Since the classical homotopy equivalence data \eqref{eq:ClassicalRetractDTay} 
is no special deformation retract, the perturbed one 
is also no special one. But it still has some nice properties.

\begin{proposition}
\label{prop:PartialHomDefNearlySpecial}
  One has 
	\begin{equation}
		p_\hbar \circ h_\hbar
		=
		0
		=
		h_\hbar \circ i_\hbar
		\quad \quad \text{ and }\quad\quad
		p_\hbar \circ i_\hbar 
		=
		\id.
	\end{equation}
\end{proposition}
\begin{proof}
The properties follow from $p\circ h=0= h\circ i$, $p\circ i = \id$  
and $\widetilde{\Phi}^2=0$.
\end{proof}
 
Thus the deformation retract \eqref{eq:defomedRetractDtay} satisfies all properties of 
a special deformation retract except for $h_\hbar\circ h_\hbar=0$, and we can still apply the 
homotopy transfer theorem.

\subsection{Application of the Homotopy Transfer Theorem}
\label{subsec:AppHTTDpoly}

We use the homotopy transfer theorem to 
extend $p_\hbar$ to an $L_\infty$-morphism. We denote the $L_\infty$-structure on 
the Taylor expansion by $Q$ and the extension of 
$h_\hbar$ to the symmetric algebra as in \eqref{eq:extendedHomotopy} by $H$. 
Then applying the homotopy transfer theorem in the form of 
Theorem~\ref{thm:HTTJonas} to the deformation retract 
\eqref{eq:defomedRetractDtay} induces higher brackets $(Q_C)_k^1$
on $\left( \prod_{i=0}^\infty (\Sym^i \liealg{g} \otimes \Dpoly(C))\right)^\group{G}[[\hbar]]$:

\begin{proposition}
\label{prop:HTTDTayCartan}
The maps
  \begin{align}
	\label{eq:HTTCartanDiffopDiff}
		  (Q_C)_1^1
			=
			- \del_\hbar,
			\quad \quad\quad\quad\;
			(Q_C)^1_{k+1}
			=
			P^1_k \circ Q^k_{k+1} \circ i_\hbar^{\vee (k+1)},
	\end{align}
	where
	\begin{equation}
	  P_1^1
		=
		p_\hbar
		=
		p, 
		\quad\quad\quad\quad\;\;
		P^1_{k+1}
		=
		\left(\sum_{\ell = 2}^{k+1} 
		Q_{C,\ell}^1 \circ P^\ell_{k+1} -  
    P_k^1 \circ 	Q^k_{k+1}\right)\circ H_{k+1}
		\quad \text{ for  } k\geq 1,
	\end{equation}
	induce a codifferential $Q_C$ on the symmetric coalgebra of 
	$\left( \prod_{i=0}^\infty (\Sym^i \liealg{g} \otimes 
			\Dpoly(C))\right)^\group{G}[[\hbar]][1]$ and 
	an $L_\infty$-quasi-isomorphism $P$
	\begin{equation}
	  P
		\colon
		\left(D_\Tay(C\times\liealg{g}^*)[[\hbar]],[\star_G-J,\argument],
		[\argument,\argument]\right) 
		\longrightarrow
		\left(\left( \prod_{i=0}^\infty (\Sym^i \liealg{g} \otimes 
			\Dpoly(C))\right)^\group{G}[[\hbar]],Q_C\right) .
	\end{equation}
\end{proposition}
\begin{proof}
The proposition follows directly from homotopy transfer theorem as in Theorem~\ref{thm:HTTJonas}. 
Note that we do indeed not need $h_\hbar\circ h_\hbar =0$, only the other properties of 
a special deformation retract from Proposition~\ref{prop:PartialHomDefNearlySpecial}.
\end{proof}

Let us take a closer look at the higher brackets $Q_C$ 
induced by the homotopy transfer theorem. One can check that they vanish:

\begin{proposition}
\label{prop:HigherBracketsDtayvanish}
  One has 
  \begin{equation*}
    (Q_C)^1_{k+1} 
	  =
	  0
	  \quad \quad \quad 
	  \forall \, k \geq 2.
  \end{equation*}
\end{proposition}
\begin{proof}
In the higher brackets with $k\geq 2$ one has 
\begin{equation*}
  H_k \circ Q^k_{k+1} \circ i_\hbar^{\vee (k+1)},
\end{equation*}
where in $H_k$ one component consists of the application of $\widetilde{\Phi}$, 
i.e. contains an insertion of a linear coordinate function $e_t$. We claim that it has to vanish. 
At first, it is clear that the image of $i$ vanishes if one argument is $e_t$. 
Let us now show that $i_\hbar$ satisfies the same property, which directly gives the 
proposition since then also the bracket vanishes if one inserts a $\liealg{g}^*$-coordinate. 

We can compute for homogeneous $D \in D_\Tay^d(C\times\liealg{g}^*)$ 
and $f_0,\dots, f_d \in \prod_i(\Sym^i\liealg{g}\tensor\Cinfty(C))$
\begin{align*}
  \Phi \circ & B (D) (f_0,\dots, f_d)
	= 
	\sum_{t=1}^n\sum_{i\leq j<d+1}(-1)^i (B( D))\left(f_0,\dots,f_{i-1},e_t,\dots,
	\frac{\del}{\del e_t} f_j,\dots, f_{d}\right)  \\
	& = 
	\sum_{t=1}^n\sum_{i\leq j<d+1}(-1)^i \big(
	\hbar m_G ( f_0,  D(f_1,\dots,f_{i-1},e_t,\dots,
	\frac{\del}{\del e_t} f_j,\dots, f_{d}))  \\
	& \quad
	-  D(\hbar m_G(f_0,f_1),\dots,f_{i-1},e_t,\dots,
	\frac{\del}{\del e_t} f_j,\dots, f_{d}) + \cdots \\
	& \quad	+ (-1)^{d} \hbar m_G( D(f_0,\dots,f_{i-1},e_t,\dots,
	\frac{\del}{\del e_t} f_j,\dots, f_{d-1}),f_d)
	\big).
\end{align*}
If $D$ vanishes if one of the arguments 
is a $\liealg{g}^*$-coordinate, then this simplifies to
\begin{align*}
  \Phi \circ  B&  (D) (f_0,\dots, f_d)
	= 
	\sum_{j=0}^d \big(
	\hbar m_G ( e_t,  D(f_0,\dots,f_{i-1},\dots,
	\frac{\del}{\del e_t} f_j,\dots, f_{d}))  \\
	& \quad
	-  D(\hbar m_G(e_t,f_0),\dots,f_{i-1},\dots,
	\frac{\del}{\del e_t} f_j,\dots, f_{d}) \big)  \\
	& \quad
	+ \sum_{j=1}^d D(\hbar m_G(f_0, e_t),\dots,f_{i-1},\dots,
	\frac{\del}{\del e_t} f_j,\dots, f_{d}) + \dots,
\end{align*}
where $e_t$ is always an argument of $\hbar m_G$. In particular, we 
know $\hbar m_G(e_i,e_j)=\frac{\hbar}{2} [e_i,e_j]$ and we see that 
the above sum vanishes if one of the functions $f_i$ is a $\liealg{g}^*$-coordinate, i.e. 
$\Phi \circ B (D)$ has the same vanishing property as $D$. 
The same holds for $\widetilde{\Phi} \circ B(D)$ and this shows by induction 
that the image of $i_\hbar$ has the same property and the proposition is shown.
%
%
\end{proof}


Considering $(Q_C)^1_2$, we can simplify \eqref{eq:HTTCartanDiffopDiff} to
\begin{align*}
  (Q_C)^1_2
	& =	
	\sum_{k=1}^\infty p \circ Q^1_2 \circ 
	\left( (\widetilde{\Phi}  \circ B)^k\circ i \vee i + 
	i \vee (\widetilde{\Phi}  \circ B)^k\circ i  \right)
	+ p \circ Q^1_2 \circ (i\vee i),
	\end{align*}
where the last term is the usual Gerstenhaber bracket. This is clear since 
$\widetilde{\Phi} $ adds a differential in $\liealg{g}^*$-direction and 
the bracket can only eliminate 
it on one argument. Recall that we also have the canonical projection 
$\pr \colon \left( \prod_{i=0}^\infty (\Sym^i \liealg{g} \otimes 
\Dpoly(C))\right)^\group{G} \rightarrow \Dpoly(M_\red)$ which projects 
first to symmetric degree zero and then restricts to $\Cinfty(C)^\group{G}\cong 
\Cinfty(M_\red)$. It is a 
DGLA morphism with respect to classical structures, i.e. Hochschild 
differentials and Gerstenhaber brackets. We extend it $\hbar$-linearly and 
can show that it is also a DLGA morphism 
with respect to the deformed DGLA structure $Q_C$:

\begin{proposition}
\label{prop:prProjectionfromTransferredStructure}
  The projection induces a DGLA morphism
  \begin{equation}
    \label{eq:PrFromCartanDGLADpoly}
		\pr \colon 
		\left(\left( \prod_{i=0}^\infty 
		(\Sym^i \liealg{g} \otimes \Dpoly(C))\right)^\group{G}[[\hbar]], Q_C\right)
		\longrightarrow 
		(\Dpoly(M_\red)[[\hbar]],\del, [\argument,\argument]_G).
  \end{equation}
\end{proposition}
\begin{proof}
By the explicit form of the differential $(Q_C)^1_1 = -\del_\hbar = -(\del + \delta)$ from 
Proposition~\ref{prop:DtayHomPerDiff} we 
know that $\pr \circ \del_\hbar = \pr \circ \del = \del \circ \pr$. Thus it only remains to show 
that we have $\pr \circ (Q_C)^1_2 = Q^1_2 \circ \pr^{\vee 2}$, which is equivalent to showing 
\begin{align}
\tag{$*$}
\label{eq:HigherBinaryBracket}
  \pr \circ \sum_{k=1}^\infty p \circ Q^1_2 \circ 
	\left( (\widetilde{\Phi}  \circ B)^k\circ i \vee i + 
	i \vee (\widetilde{\Phi}  \circ B)^k\circ i  \right)
	=
	0.
\end{align}
In the proof of Proposition~\ref{prop:HigherBracketsDtayvanish} we computed 
$\Phi \circ B (D)$ of some $D \in D_\Tay^d(C\times\liealg{g}^*)$ and 
we saw that the image of $i$ vanishes if one inserts a $\liealg{g}^*$-coordinate 
and that $\Phi \circ B$ preserves this property. Therefore, we 
got for such a $D$ that vanishes if one of the arguments is $e_t$ 
\begin{align}
\tag{$**$}
\label{eq:ExpressionPhiB}
\begin{split}
  \Phi \circ  B&  (D) (f_0,\dots, f_d)
	= 
	\sum_{j=0}^d \big(
	\hbar m_G ( e_t,  D(f_0,\dots,f_{i-1},\dots,
	\frac{\del}{\del e_t} f_j,\dots, f_{d}))  \\
	& \quad
	-  D(\hbar m_G(e_t,f_0),\dots,f_{i-1},\dots,
	\frac{\del}{\del e_t} f_j,\dots, f_{d}) \big)  \\
	& \quad
	+ \sum_{j=1}^d D(\hbar m_G(f_0, e_t),\dots,f_{i-1},\dots,
	\frac{\del}{\del e_t} f_j,\dots, f_{d}) - \dots  \\
	& \quad 
	-D(f_0,\dots, ,f_{d-1},\hbar m_G(e_t, \frac{\del}{\del e_t} f_d)),
\end{split}
\end{align}
where $f_0,\dots, f_d \in \prod_i(\Sym^i\liealg{g}\tensor\Cinfty(C))$. 
Let us consider now \eqref{eq:HigherBinaryBracket} applied to homogeneous 
$P\otimes D \vee Q \otimes D'$, 
where $P,Q \in \Sym \liealg{g}$ and $D,D' \in \Dpoly(C)[[\hbar]]$. 
At first we note that this is zero if both 
$P \neq 1 \neq Q$ since the 
Gerstenhaber bracket can cancel at most one term. Similarly, it is zero if both $P =1 =Q$.
Thus we consider w.l.o.g. $ D, Q\otimes D'$ with $Q \neq 1$ and 
$D \in (D_\poly^d(C))^{\group{G}}[[\hbar]]$, where the only possible contributions 
are
\begin{align*}
  \pr \circ p \circ Q^1_2 \left( 
	( (\widetilde{\Phi} \circ B)^k D)\vee (Q \otimes D')\right)
	=
	 (-1)^{d + (dd')}   \pr \circ p 
	\left( ( (\widetilde{\Phi} \circ B)^k D)\circ  (Q \otimes D')\right)	
\end{align*}
for all $k\geq 1$. Note that, up to a sign, this is  
$((\widetilde{\Phi} \circ B)^k D)\circ (Q \otimes D')$ 
applied to 
invariant functions $\Cinfty(C)^\group{G}[[\hbar]]$ and then projected 
to $\Sym^0\liealg{g}$. But on invariant functions 
the vertical vector fields and the differentials in 
$\liealg{g}^*$-direction vanish, and we have only one slot where they can give a non-trivial contribution, namely $Q\otimes D'$. We fix the symmetric degree 
$Q \in \Sym^i\liealg{g}$ and get 
\begin{align*}
  \pr \circ  p \circ Q^1_2 \left( 
	( (\widetilde{\Phi}  \circ B)^k D)\vee (Q \otimes D')   \right) 
	& =
	\frac{(-1)^{d + (dd')}}{i}    \pr \circ p \left(
	(\Phi (B (\widetilde{\Phi} B)^{k-1} D)_{i} )
	\circ (Q \otimes D' )\right)   \\
	& = 
	\frac{(-1)^{d + (dd')}}{i}   \pr \circ p \left(
	(\Phi B (\widetilde{\Phi} B)^{k-1} D) \circ 
  ( Q \otimes D'	)\right).  
\end{align*}
Here $(B (\widetilde{\Phi} B)^{k-1} D)_{i}$ denotes the component of 
$B (\widetilde{\Phi} B)^{k-1} D$ with $i$ differentiations in $\liealg{g}^*$-direction. 
The $1/i$ comes from the degree of the homotopy \eqref{eq:totalhomotopy} since we 
have no $\Sym \liealg{g}^*$-degree and since the only term that can be non-trivial 
is the one with $i$ differentiations in $\liealg{g}^*$-direction applied to $Q$.
We compute with \eqref{eq:ExpressionPhiB}
\begin{align*}
  \pr \circ  p \circ &  Q^1_2 \left( 
	( (\widetilde{\Phi}  \circ B)^k D)\vee (Q \otimes D')   \right)   
	=
	\frac{(-1)^{d + (dd')}}{i}   \pr \circ p \left(
	(\Phi B (\widetilde{\Phi} B)^{k-1} D) \circ 
  (Q \otimes D')	\right) \\
	& =
	\frac{(-1)^{d + (dd')}}{i}  \pr \circ p  \bigg(
	\left( - \hbar \Lie_{(e_t)_C} \circ \pr\at{\Sym^0\liealg{g}} (\widetilde{\Phi} \circ B)^{k-1} D \circ 
	\frac{\del}{\del e_t}\right) \circ (Q\otimes D')  \\
	& \quad - 
	\left( \pr\at{\Sym^0\liealg{g}} (\widetilde{\Phi} \circ B)^{k-1} D \circ 
	(\hbar m_G(e_t, \frac{\del}{\del e_t}\argument))\right) \circ (Q\otimes D')  \bigg)  \\
	& =
	\frac{(-1)^{d + (dd')}}{i}   \pr \circ p  \bigg( 
	\left( - \hbar \Lie_{(e_t)_C} \circ \pr\at{\Sym^0\liealg{g}} (\widetilde{\Phi} \circ B)^{k-1} D \right) 
	\circ (\frac{\del}{\del e_t}Q\otimes D')  \\
	& \quad - 
	\left( \pr\at{\Sym^0\liealg{g}} (\widetilde{\Phi} \circ B)^{k-1} D \right) \circ 
	( (\hbar m_G(e_t,\frac{\del}{\del e_t}\argument) )\circ (Q\otimes D') ) \bigg).
\end{align*}
But we know $\hbar m_G(e_t,\argument) = - \hbar \Lie_{(e_t)_C} + \hbar m_{\liealg{g}}(e_t, 
\argument)$, where $\hbar m_{\liealg{g}}$ denotes the higher components of the Gutt product 
on $\liealg{g}^*$. Moreover, we have by the invariance 
\begin{align*}
  - [\Lie_{(e_t)_C}, \pr\at{\Sym^0\liealg{g}} (\widetilde{\Phi} \circ B)^{k-1} D]_G
	=
	\left[ - f_{t k}^j e_j \frac{\del}{\del e_k}, \pr\at{\Sym^0\liealg{g}} (\widetilde{\Phi} \circ B)^{k-1} D
	\right]_G
\end{align*}
and thus
\begin{align*}
  \hbar \pr \circ  & p  \left( 
	\left(- [\Lie_{(e_t)_C}, \pr\at{\Sym^0\liealg{g}} 
	(\widetilde{\Phi} \circ B)^{k-1} D]_G \right)
	\circ (\frac{\del}{\del e_t} Q \otimes D')\right) \\
	& 
	=
	 \hbar\pr \circ p \left( 
	\left( \pr\at{\Sym^0\liealg{g}} (\widetilde{\Phi} \circ B)^{k-1} D \circ 
	\left(f_{t k}^j e_j \frac{\del}{\del e_k}\right)
	\right) \circ (\frac{\del}{\del e_t}Q\otimes D')\right)  \\
	& =
	\hbar	\pr \circ p 
	\left( \left(\pr\at{\Sym^0\liealg{g}} (\widetilde{\Phi} \circ B)^{k-1} D \right) 
	\circ (f_{t k}^j e_j \frac{\del}{\del e_k}\frac{\del}{\del e_t} Q\otimes D')\right)
	=
	0.
\end{align*}
The only remaining terms are 
\begin{align*}
  \pr \circ & p \circ Q^1_2 \left( 
	(  (\widetilde{\Phi} \circ B)^k D)\vee (Q \otimes D')\right)  
	=
	(-1)^{d + (dd')}   \pr \circ p 
	\left( (\pr\at{\Sym^0\liealg{g}}  (\widetilde{\Phi} \circ B)^k D)\circ  
	(Q \otimes D')\right) \\
	& =
	 - \frac{ (-1)^{d + (dd')}}{i}\pr \circ p  \bigg(
	\left( \pr\at{\Sym^0\liealg{g}} (\widetilde{\Phi} \circ B)^{k-1} D \right) 
	\circ  \left(\hbar m_{\liealg{g}}(e_t\frac{\del}{\del e_t} Q) \otimes D'\right) \bigg).
\end{align*}
We know that $\hbar m_{\liealg{g}}(e_t, \frac{\del}{\del e_t} Q)$ is either zero or 
in $\Sym^{>0} \liealg{g}$ and the statement follows by induction.
\end{proof}

In particular, we can compose this projection $\pr$ with the $L_\infty$-projection 
from Proposition~\ref{prop:HTTDTayCartan} that we constructed with the homotopy transfer 
theorem. Summarizing, we have shown:

\begin{theorem}
\label{thm:NotTwistedRedMorph}
  There exists an $L_\infty$-morphism
	\begin{equation}
		\label{eq:DredTay}
		D_\red
		=
		\pr \circ P \colon 
		\left(D_\Tay(C\times\liealg{g}^*)[[\hbar]],[\star_G-J,\argument],
		[\argument,\argument]\right) 
		\longrightarrow
		(\Dpoly(M_\red)[[\hbar]],\del, [\argument,\argument]_G).
	\end{equation}
\end{theorem}

Finally, as in 
 the polyvector field case in \cite{esposito.kraft.schnitzer:2020a:pre},
we can twist the above morphism to obtain 
an $L_\infty$-morphism from the curved equivariant polydifferential operators into the 
Cartan model and therefore also into the polydifferential operators on $M_\red$, see Proposition~\ref{prop:twistinglinftymorphisms} for the basics 
of the twisting procedure.

\begin{proposition}
  Twisting the reduction $L_\infty$-morphism $D_\red$ from \eqref{eq:DredTay} with $-\hbar m_G$ 
	yields an $L_\infty$-morphism
	\begin{equation}
		\label{eq:DredTaycurved}
		D_\red^{-\hbar m_G}
		\colon 
		\left(D_\Tay(C\times\liealg{g}^*)[[\hbar]],\hbar\lambda, \del + [-J,\argument],
		[\argument,\argument]\right) 
		\longrightarrow
		(\Dpoly(M_\red)[[\hbar]],\del, [\argument,\argument]_G),
	\end{equation}
	where $\lambda = e^i \otimes (e_i)_M$ denotes the curvature.	
\end{proposition}
\begin{proof}
At first we check that the curvature is indeed given by
\begin{align*}
  e^i \otimes [- e_i, -\hbar m_G ]_G
	& =
	e^i \otimes -[e_i, \argument ]_{\star_G}
	=
	e^i \otimes ( \hbar \Lie_{(e_i)_C} - \hbar \ad(e_i))
	=
	\hbar \lambda,
\end{align*}
compare Lemma~\ref{lem:GuttProduct}.
The only thing left to show is that the DGLA structure on $M_\red$ is not changed, 
which is equivalent to
\begin{align*}
  \sum_{k=1}^\infty \frac{(-\hbar)^k}{k!} (D_\red)^1_k (m_G \vee \cdots \vee m_G)
	=
	0.
\end{align*}
But using the explicit form of $P$ 
from Proposition~\ref{prop:HTTDTayCartan} we see inductively that 
$P$ vanishes if every argument has a differential in 
$\liealg{g}^*$-direction and the statement is shown.
\end{proof}

\begin{remark}
  In the polyvector field case from 
	\cite[Proposition~4.29]{esposito.kraft.schnitzer:2020a:pre} we 
	saw that the structure maps of the twisted morphism coincide with the structure maps 
	of the original one. In our case it is not clear, i.e. one might indeed have 
	$D_\red^{-\hbar m_G} \neq D_\red$.
\end{remark}

This reduction morphism can be used to obtain a reduction morphism of the equivariant 
polydifferential operators $D_\liealg{g}^\bullet(M)$ of more general manifolds 
$M \neq C\times \liealg{g}^*$. More explicitly, assuming that the action 
is proper around $C$ and free on $C$, we can restrict at first  
to $M_\nice \cong U_\nice \subset C\times\liealg{g}^*$, i.e. we have
\begin{align*}
	\cdot\at{U_\nice }\colon  
	(D_{\liealg{g}}(M)[[\hbar]],\hbar\lambda,\del^\liealg{g}-[J,\argument]_{\liealg{g}},
	[\argument,\argument]_{\liealg{g}})
	\longrightarrow
	 (D_{\liealg{g}}(U_\nice )[[\hbar]],
	 \hbar\lambda\at{U_\nice },\del^\liealg{g}-[J\at{U_\nice },\argument]_{\liealg{g}},
	[\argument,\argument]_{\liealg{g}}).
\end{align*}
But on $U_\nice$ we can perform the Taylor expansion that is a morphism of 
curved DGLAs
	\begin{align*}
	   D_{\liealg{g}^*}\colon 
	   (D_{\liealg{g}}(U_\nice )[[\hbar]],
	 	\hbar\lambda\at{U_\nice },\del^\liealg{g}-[J\at{U_\nice },\argument]_{\liealg{g}},
	 	[\argument,\argument]_{\liealg{g}})
	 	\rightarrow
	 	\big(D_\Tay (C\times \liealg{g}^*)[[\hbar]],\hbar\lambda,\del-[J,\argument],
	 	[\argument,\argument]\big).
\end{align*}
Finally, we can compose it with $D_\red^{-\hbar m_G}$ and obtain 
the following statement:

\begin{theorem}
  \label{thm:GlobalDred}
  The composition of the above morphisms results in an $L_\infty$-morphism 
\begin{equation}
\label{eq:GlobalDred}
  \mathrm{D}_\red \colon
  (D_{\liealg{g}}(M)[[\hbar]],\hbar\lambda,\del^\liealg{g}-[J,\argument]_\liealg{g}, 
	[\argument,\argument]_\liealg{g})
	\longrightarrow 
	(\Dpoly(M_\red)[[\hbar]],0,
\del,[\argument,\argument]_G),
\end{equation}	   
	called \emph{reduction $L_\infty$-morphism}. 
\end{theorem}

\begin{remark}[Choices]
Note that the only non-canonical choice we made is a open neighbourhood of $C$ 
in $M$ which is diffeomorphic to a star shaped open neighbourhood of $C$ in $C\times 
\liealg{g*}$.  
Recall that the choice of this neighbourhood works as follows:
Take an arbitrary $\group{G}$-equivariant tubular neighbourhood embedding 
$\psi\colon \nu(C)\to U\subseteq M$, where $\nu(C)$ denotes the normal bundle. 
Then define  
\begin{align*}
\phi\colon \nu(C)\ni [v_p]\longmapsto (p,J(\psi([v_p]))\in C\times \liealg{g}^*
\end{align*}
which is close to $C$ a diffeomorphism. After some suitable restriction 
we obtain the identification. Nevertheless, we had to choose a 
$\group{G}$-equivariant tubular neighbourhood and any two choices differ by a
$\group{G}$-equivariant 
local diffeomorphism around $C$
	\begin{align*}
	A\colon C\times\liealg{g}^*\longrightarrow C\times\liealg{g}^*
	\end{align*}	 
which is, restricted to $C$ the identity. One can show that one has in the 
Taylor expansion
	\begin{align*}
	 D_{\liealg{g}^*} (A^*f)=\E^{X}  D_{\liealg{g}^*}(f)
	\end{align*}
for a vector field field $X\in\prod_{i\geq 1}  (\Sym^i\liealg{g}\tensor \mathfrak{X}(C))^\group{G}
\subseteq D_\Tay(C\times \liealg{g}^*)$.  Since any vector field  is closed,
$X$ does not derive in $\liealg{g}^*$-direction and $\lambda$ is central, 
we obtain an inner automorphism
	\begin{align*}
	\E^{[X,\argument]}
	\colon 
	\big(D_\Tay (C\times \liealg{g}^*)[[\hbar]],\hbar\lambda,\del-[J,\argument],
	 	[\argument,\argument]\big)
	\longrightarrow
	\big(D_\Tay (C\times \liealg{g}^*)[[\hbar]],\hbar\lambda,\del-[J,\argument],
	 	[\argument,\argument]\big)
	\end{align*}
of curved Lie algebras 
which acts trivially on the level of equivalence classes of Maurer-Cartan 
elements. 
We are moreover certain, that the two reduction $L_\infty$-morphisms are homotopic 
in a suitable curved setting, which, to our knowledge, is not developed yet.   
\end{remark}

As a last remark of this section, we want to mention a very interesting observation,
which is not directly connected to the rest of this paper and/or the results of it. 
Nevertheless, we felt that it can be interesting from many other different 
perspectives.

\begin{remark}[Cartan model]
One can show that DGLA structure $Q_C$ from Proposition \ref{prop:HTTDTayCartan} on 
$\prod_{i=0}^\infty (\Sym^i \liealg{g} \otimes \Dpoly(C))^\group{G}[[\hbar]]$ 
restricts to $(\Sym \liealg{g} \otimes \Dpoly(C))^\group{G}[\hbar]$ and hence 
can be evaluated at $\hbar =1$. Moreover, we still have the DGLA map 	
\begin{align*} 	
  \pr\colon (\Sym \liealg{g} \otimes \Dpoly(C))^\group{G} 	
	\longrightarrow 	
	\Dpoly(M_\red). 	
\end{align*} 
We want to sketch the proof of the fact that this is a quasi-isomorphism, 
which motivates us to interpret $(\Sym \liealg{g} \otimes \Dpoly(C))^\group{G}$ 
as a Cartan model for equivariant polydifferential operators, generalizing the 
Cartan model for equivariant polyvector fields from 
\cite[Section~4.2]{esposito.kraft.schnitzer:2020a:pre}. 

Picking a $\group{G}$-invariant 
covariant derivative (not necessarily torsion-free) for which the fundamental vector 
fields are flat in fiber direction one can, using the PBW-ismorphism for Lie
 algebroids (see \cite{Xu} and \cite{nistor.weinstein.xu:1999a})  
prove that there is an equivariant cochain map $K\colon \Dpoly(C)\to \Tpoly(C)$ and 
an equivariant homotopy $h\colon \Dpoly^\bullet(C)\to  \Dpoly^{\bullet-1}(C)$, such that    
\begin{equation}    
  \label{eq:DefRetHoch} 	 
	\begin{tikzcd}       
	\Tpoly(C) 			\arrow[rr, shift left, "\hkr"]      
	&&  			
	(\Dpoly(C),\partial) 			
	\arrow[loop right, 	distance=3em, start anchor={[yshift=1ex]east}, end anchor={[yshift=-1ex]east}]{}{h} 			
	\arrow[ll, shift left, "K"] 	  
	\end{tikzcd}   
	\end{equation} 
is a special deformation retract. Additionally, one can show that 	
\begin{align*} 	
  K(D_1\cup D_2)=K(D_1)\wedge K(D_2) 	
	\ 	\text{ and } 	\ 	K(\Lie_P)= 	
	\begin{cases} 	-P_C, & \text{ for } P\in \liealg{g}\subseteq \Sym\liealg{g}\\ 	
	0, & \text{ else} 	
	\end{cases} 	
\end{align*} 
for $D_1,D_2\in\Dpoly(C)$ and $P\in\Sym\liealg{g}$. We extend now \eqref{eq:DefRetHoch} to    
\begin{equation*} 	 
\begin{tikzcd}      
 ((\Sym\liealg{g}\otimes\Tpoly(C))^\group{G},0) 			
  \arrow[rr, shift left, "\hkr"]       &&  			
	((\Sym\liealg{g}\otimes\Dpoly(C))^\group{G},\partial) 			
	\arrow[loop right, 	distance=3em, start anchor={[yshift=1ex]east}, end anchor={[yshift=-1ex]east}]{}{h} 
	\arrow[ll, shift left, "K"] 	  
\end{tikzcd}  
\end{equation*} 
to obtain a special deformation retract. Now we include $\delta$ as in Proposition~\ref{prop:DtayHomPerDiff} 
and see it as a pertubation of $\partial$. One can show that the pertubation is small in the sense 
of the homological perturbation lemma as in \cite{crainic:2004a:pre}, and we obtain   
\begin{equation*} 	 
\begin{tikzcd}      
 ((\Sym\liealg{g}\otimes\Tpoly(C))^\group{G},\delta) 			
 \arrow[rr, shift left, "\widehat{\hkr}"]       
 &&  			
((\Sym\liealg{g}\otimes\Dpoly(C))^\group{G},\partial+\delta) 			
\arrow[loop right, 	distance=3em, start anchor={[yshift=1ex]east}, end anchor={[yshift=-1ex]east}]{}{\widehat{h}} 			
\arrow[ll, shift left, "\widehat{K}"] 	 
\end{tikzcd}  
\end{equation*} 
where $\delta$ is the differential 
\begin{align*} 
  \delta(P\otimes X)
	=
	\ins(e^{i})P\tensor (e_i)_C\wedge X 
\end{align*}  
obtained in \cite[Definition~4.14]{esposito.kraft.schnitzer:2020a:pre} 
on $(\Sym\liealg{g}\otimes\Tpoly(C))^\group{G}$. 
Finally, one can show that    
\begin{center}  	
\begin{tikzcd}  	
  ((\Sym\liealg{g} \tensor \Tpoly(C))^\group{G},\delta)  	 
  \arrow[rr, shift left, "\widehat{\hkr}"]  	 \arrow[dd]  	&&  	
	((\Sym\liealg{g} \tensor \Dpoly(C))^\group{G},\partial+\delta)  	\arrow[dd]  	\\  
		&&  	\\  	(\Tpoly(M_\red),0)  	\arrow[rr, shift left, "\hkr"]  	
		&&  	 (\Dpoly(M_\red),\partial)\\  	
\end{tikzcd}  	
\end{center}
commutes and both of the horizontal maps are quasi-isomorphisms as well 
as the left vertical one which implies the claim. 
\end{remark}

\section{Comparison of the Reduction Procedures}
\label{sec:ComparisonofRedProd}

At the level of Maurer-Cartan elements, we know that the 
$L_\infty$-morphism $\mathrm{D}_\red$ from Theorem~\ref{thm:GlobalDred} 
induces a map from equivariant star products $(\star,H)$ with 
quantum momentum map $H = J + O(\hbar)$ on $M$ to star products $\star_\red$ on 
the reduced manifold $M_\red$. We conclude with a comparison of this 
reduction procedure with the reduction of formal Poisson structures via the 
quantized Koszul complex as in \cite{bordemann.herbig.waldmann:2000a,gutt.waldmann:2010a}, see also our adapted version in 
Appendix~\ref{sec:BRSTReductionStarProducts}.

We assume for simplicity $M= C\times \liealg{g}^*$ and work in the Taylor 
expansion of the equivariant polydifferential operators. Moreover, 
we identify $\Cinfty(C)$ with $\prol \Cinfty(C)\subset 
\Cinfty(C\times\liealg{g}^*)$. Let us start with an equivariant star 
product $(\star,H= J + \hbar H')$ on $C\times \liealg{g}^*$, which means 
that $\hbar \pi_\star -\hbar H' = \star - \star_G - (H-J)$ is 
Maurer-Cartan element in 
\begin{equation*}
				   \left(D_\Tay(C\times\liealg{g}^*)[[\hbar]],
					[\star_G-J,\argument],[\argument,\argument]\right).
\end{equation*}

\begin{proposition}
Defining $I_1^1 = i_\hbar$ and 
$I_k^1 = h_\hbar \circ  Q_{2}^1 \circ I^2_{k+1}$ gives an $L_\infty$-morphism
\begin{equation}
	  I
		\colon
			\left(\left( \prod_{i=0}^\infty (\Sym^i \liealg{g} \otimes 
			\Dpoly(C))\right)^\group{G}[[\hbar]],Q_C\right)
		\longrightarrow
		\left(D_\Tay(C\times\liealg{g}^*)[[\hbar]],[\star_G-J,\argument],
		[\argument,\argument]\right) .
\end{equation}
Moreover, one $I$ is a quasi-inverse of the $L_\infty$-projection 
$P$ from Proposition~\ref{prop:HTTDTayCartan} and one has 
$P \circ I = \id$.
\end{proposition}
\begin{proof} 
Note that we have in general  $h_\hbar^2 \neq 0$, but the only part 
of the homotopy that appears in the above recursions is $\widetilde{\Phi}$, 
where we know $\widetilde{\Phi}\circ\widetilde{\Phi} = 0$. 
Therefore, the statement follows from Proposition~\ref{prop:Infinityinclusion}.
\end{proof}

We get with Corollary~\ref{cor:IPsimId}:

\begin{corollary}
  The $L_\infty$-morphism $I$ is compatible with the filtration induced 
	by $\hbar$ and 
	\begin{equation}
	\hbar \widetilde{\pi}_\star 
	= (I\circ P)^1(\cc{\exp}(\hbar \pi_\star-\hbar H'))
	\in
	\left(D_\Tay(C\times\liealg{g}^*)[[\hbar]],
	[\star_G-J,\argument],[\argument,\argument]\right)
	\end{equation}
	is a well-defined 
	Maurer-Cartan element that is equivalent 
	to $\hbar\pi_\star- \hbar H'$. 
	In particular, $(\widetilde{\star}=\star_G + \hbar\widetilde{\pi}_\star,J)$ 
	is a strongly invariant star product, i.e. an equivariant star product s.t. 
	the quantum momentum map is just the classical momentum map, 
	and it is equivariantly equivalent 
	to $(\star,H)$.
\end{corollary}

The reduction of $(\widetilde{\star},J)$ via the 
reduction $L_\infty$-morphism $D_\red$ is now easy:

\begin{lemma}
  The reduction $L_\infty$-morphism
	\begin{equation}
		D_\red
		=
		\pr \circ P \colon 
		\left(D_\Tay(C\times\liealg{g}^*)[[\hbar]],[\star_G-J,\argument],
		[\argument,\argument]\right) 
		\longrightarrow
		(\Dpoly(M_\red)[[\hbar]],\del, [\argument,\argument]_G)
	\end{equation}
	from Theorem~\ref{thm:NotTwistedRedMorph} maps 
	$\hbar \widetilde{\pi}_\star$ to a Maurer-Cartan element 
	$\hbar m_\red = \pr \circ P^1(\exp \hbar \widetilde{\pi}_\star)$ in the 
	polydifferential operators on $M_\red$. The corresponding star product 
	$\widetilde{\star}_\red = \mu + \hbar m_\red$ is given by
  \begin{equation}
    \label{eq:DredStarRedExpl}
		\pr^*(u_1 \widetilde{\star}_\red u_2) 
  	=
		\iota^* (\prol(\pr^*u_1)\widetilde{\star} \prol(\pr^* u_2))
  \end{equation}
	for all $u_1,u_2\in\Cinfty(M_\red)[[\hbar]]$.
\end{lemma}
\begin{proof}
By definition of $\hbar \widetilde{\pi}_\star$ we know 
$h_\hbar \hbar \widetilde{\pi}_\star=	\widetilde{\Phi}(\hbar \widetilde{\pi}_\star)=0$, and thus
\begin{equation*}				 
  \hbar m_\red
	=
	\pr \circ P^1(\exp \hbar \widetilde{\pi}_\star)
	=
	\pr \circ p(\hbar \widetilde{\pi}_\star).
 \end{equation*}
Equation~\eqref{eq:DredStarRedExpl} follows since 
$\hbar m_G(\prol(\pr^*u_1), \prol(\pr^* u_2))=0$.
\end{proof}

Moreover, we know by Lemma~\ref{lemma:BoldiotaClassical} that the 
BRST reduction of $ \mu + \hbar m_G	+ \hbar \widetilde{\pi}_\star$ 
coincides with \eqref{eq:DredStarRedExpl}, and we have shown:

\begin{theorem}
\label{thm:CompRed}
  Let $(\star,H)$ be an equivariant star product on $M$. 
	Then the reduced star product induced by 
	$D_\red$ from Theorem~\ref{thm:NotTwistedRedMorph} and the 
	reduced star product 	via the formal Koszul complex 
	\eqref{eq:ReducedStarProduct} are equivalent. 
\end{theorem}
\begin{proof}
We know that both reduction procedures map equivalent equivariant 
star products to equivalent reduced star products. Moreover, we saw above 
that both reduction procedures coincide on $(\widetilde{\star}=\star_G + \hbar\widetilde{\pi}_\star,J)$ which is equivariantly equivalent 
to $(\star,H)$.
\end{proof}

%
%
%

\appendix

\section{BRST Reduction of Equivariant Star Products}
\label{sec:BRSTReductionStarProducts}

We recall a slightly modified version of the reduction of 
equivariant star products as introduced in 
\cite{bordemann.herbig.waldmann:2000a,gutt.waldmann:2010a}, 
see also \cite{esposito.kraft.waldmann:2020a} for a discussion of this reduction scheme in the context of Hermitian star products. It relies 
on the quantized Koszul complex and the homological perturbation lemma.

\subsection{Homological Perturbation Lemma}
\label{sec:HomPerLemma}
At first we recall from \cite[Theorem~2.4]{crainic:2004a:pre} and 
\cite[Chapter~2.4]{reichert:2017b} a version of the homological perturbation lemma 
that is adapted to our setting.
Let 
\begin{equation*}
\begin{tikzcd}
  (C,\D_C) 
  \arrow[rr, shift left, "i"] 
  &&   
  (D,\D_D)  \arrow[ll, shift left, "p"]
	\arrow[loop right, distance=3em, start anchor={[yshift=1ex]east}, end anchor={[yshift=-1ex]east}]{}{h}
\end{tikzcd}
\end{equation*}
be a homotopy retract (also called homotopy equivalence data), i.e. let 
$(C,\D_C)$ and $(D,\D_D)$ be two chain complexes together with two quasi-isomorphisms
\begin{equation}
  i \colon 
  C 
  \longrightarrow 
  D 
  \quad\text{and}\quad 
  p \colon 
  D 
  \longrightarrow 
  C
\end{equation}
and a chain homotopy 
\begin{equation}
  h \colon 
  D 
  \longrightarrow 
  D 
  \quad \text{with} \quad
  \id_D - ip 
  = 
  \D_D h + h \D_D
\end{equation}
between $\id_D$ and $ip$.
Then we say that a graded map $B\colon D_\bullet \longrightarrow
D_{\bullet -1}$ with $(\D_D + B)^2=0$ is a \emph{perturbation} of the
homotopy retract. The perturbation is called \emph{small} if $\id_D + B h$ is
invertible, and the homological perturbation lemma states that in this
case the perturbed homotopy retract is a again a homotopy retract, see
\cite[Theorem~2.4]{crainic:2004a:pre} for a proof.

\begin{proposition}[Homological perturbation lemma]
  \label{lemma:homologicalperturbationlemma}
  Let
  \begin{equation*}
  \begin{tikzcd}
    (C,\D_C) 
    \arrow[rr, shift left, "i"] 
    &&   
    (D,\D_D)  \arrow[ll, shift left, "p"]
  	\arrow[loop right, distance=3em, start anchor={[yshift=1ex]east}, end anchor={[yshift=-1ex]east}]{}{h}
  \end{tikzcd}
\end{equation*}
  be a homotopy retract and let $B$ be small perturbation of $\D_D$, then the
  perturbed data
	\begin{equation}
  \begin{tikzcd}
  (C,\widehat{\D}_C)
  \arrow[rr, shift left, "I"] 
  &&   
  (D,\widehat{\D}_D) \arrow[ll, shift left, "P"]
	\arrow[loop right, distance=3em, start anchor={[yshift=1ex]east}, end anchor={[yshift=-1ex]east}]{}{H} 
  \end{tikzcd}
  \end{equation}
  with
  \begin{equation}
    \begin{alignedat}{5}
      \label{eq:deformedhomologicalmaps}
      A 
      &= 
      (\id_D + Bh)^{-1}B, \quad
      &&
      \widehat{\D}_D 
      &&= 
      \D_D + B, \quad \quad
      &&
      \widehat{\D}_C 
      &&= 
      \D_C + p Ai   ,    \\
      I
      &= 
      i -hAi ,
      &&\; 
      P 
      &&=  
      p-pAh  ,
      && 
      H 
      &&= 
      h-hAh  
    \end{alignedat}
  \end{equation}
  is again a homotopy retract.
\end{proposition}

\begin{remark}
\label{rem:DefretractsnotStable}
  In \cite{crainic:2004a:pre} it is shown that perturbations of special 
	deformation retracts are again special deformation retracts, which is in 
	general not true for deformation retracts, see 
	Section~\ref{sec:homotopytransfertheorem} for the different notions.
\end{remark}

We are interested in even simpler complexes of the following form:

\begin{equation}
  \label{eq:simplehedata}
  \begin{tikzcd}[row sep = large, column sep = large]
    0    	&
    D_0   \arrow[l]
    \arrow[d,shift left=-.75ex,swap,"p"] 
    \arrow[r,shift left=-.75ex,swap,"h_0"]      &		
    D_1   \arrow[l,shift left=-.75ex,swap,"\D_{D,1}"]
    \arrow[r,shift left=-.75ex,swap,"h_1"]  &    
    \cdots\arrow[l,shift left=-.75ex,swap,"\D_{D,2}"]  \\
    0     &		 
    C_0   \arrow[l] 
    \arrow[u,shift left=-.75ex,swap,"i"]  	&	
    0	    \arrow[l]	 	&          	
  \end{tikzcd}
\end{equation}
In this case, the perturbed homotopy retract corresponding to a small
perturbation $B$ according to \eqref{eq:deformedhomologicalmaps} is
given by
\begin{equation*}
  I 
  = 
  i, \quad \quad
  P 
  = 
  p - p(\id_D + B_1 h_0)^{-1}B_1h_0 , \quad \quad
  H 
  = 
  h-h(\id_D + Bh)^{-1}Bh
\end{equation*}
and, using the geometric power series, this can be simplified to
\begin{equation}
  \label{eq:simplifieddeformedhomologicalstuff}
  I 
  = 
  i, \quad \quad
  P 
  = 
  p(\id_D + B_1 h_0)^{-1}, \quad \quad
  H 
  = 
  h(\id_D + Bh)^{-1}.
\end{equation}	
Here we denote by $B_1\colon D_1 \longrightarrow D_0$ the degree one
component of $B$, analogously for $h$. 
By Remark~\ref{rem:DefretractsnotStable} we know that deformation retracts are 
in general not preserved under perturbations. However, in this case we see that, 
starting with a deformation retract, the additional condition $h_0 i=0$ suffices to guarantee
\begin{equation*}
  P I
	=
	p(\id_D + B_1 h_0)^{-1} i
	=
	pi
	=
	\id_{C_0}.
\end{equation*}

\subsection{Quantized Koszul Complex}
\label{sec:QuantizedKoszulComplex}
Let now $(M,\{\argument,\argument\})$ be a smooth Poisson manifold 
with a left action of the Lie group $\group{G}$.  Moreover, let 
$J\colon M \rightarrow \liealg{g}^*$ be a classical (equivariant) momentum map.
As usual, we assume that $0 \in \liealg{g}^*$ is a value and a regular value of 
$J$ and set $C= J^{-1}(\{0\})$. In addition, we require the action to be 
proper on $M$ (or at least around $C$) and free on $C$, which implies that 
$M_\red = C/ \group{G}$ is a smooth manifold. The reduction via the 
classical Koszul complex $\Anti^\bullet\liealg{g}\otimes \Cinfty(M)$ 
is one way to show that $M_\red$ is even a Poisson manifold, but we need the 
quantum version to show that we have an induced star product on $M_\red$.
The Koszul differential $\del$ is given by
\begin{equation}
  \del \colon 
  \Anti^q\liealg{g}\otimes \Cinfty(M) 
  \longrightarrow
  \Anti^{q-1}\liealg{g}\otimes \Cinfty(M) , \hspace{5pt} 
  a 
  \mapsto 
  \ins(J)a = J_{i} \insa(\basis{e}^i)a,
  \label{eq:koszuldifferential}
\end{equation}
where $\ins$ denotes the left insertion and $J = J_{i}\basis{e^i}$
the decomposition of $J$ with respect to a basis
$\basis{e}^1,\dots,\basis{e}^n$ of $\liealg{g}^*$. Then 
$\del^2 =0$ follows immediately with the commutativity of the
pointwise product in $\Cinfty(M)$. The differential $\del$ is also a
derivation with respect to associative and super-commutative product
on the Koszul complex, consisting of the $\wedge$-product on
$\Anti^\bullet \liealg{g}$ tensored with the pointwise product on the
functions.  Moreover, it is invariant with respect to the induced
$\liealg{g}$-representation
\begin{equation}
	 \label{eq:ActionofgonKoszulComplex}
	 \liealg{g}\ni \xi 
	 \mapsto 
	 \rho(\xi)
	 =
	 \ad(\xi) \otimes \id - \id \otimes \Lie_{\xi_M} \in 
	 \End(\Anti^\bullet \liealg{g}\otimes \Cinfty(M))
\end{equation}
as we have
\begin{align*}
  \del \rho(\basis{e}_a) ( x \otimes f )
	& =
	f^k_{aj} \basis{e}_k\wedge \ins(\basis{e}^j)\wedge\ins(\basis{e}^i) x \otimes J_{0,i}f 
	+ f^i_{aj}\ins(\basis{e}^j)x \otimes J_{0,i} f  \\
	& \quad
	+ \ins(\basis{e}^i)x \otimes J_{0,i} \{J_{0,a},f\}_0       \\
	& =
	\rho(\basis{e}_a)\del ( x \otimes f)
\end{align*}
for all $x \in \Anti^\bullet \liealg{g}$ and $f\in \Cinfty(M)$.

One can show that the Koszul complex is acyclic in positive degree
with homology $\Cinfty(C)$ in order zero, and that one has a
$\group{G}$-equivariant homotopy
\begin{equation}
  h\colon 
  \Anti^\bullet  \liealg{g} \otimes \Cinfty(M) 
  \longrightarrow 
  \Anti^{\bullet+1}  \liealg{g} \otimes \Cinfty(M),
\end{equation}
see \cite[Lemma~6]{bordemann.herbig.waldmann:2000a} and
\cite{gutt.waldmann:2010a}.  In other
words, this means that
\begin{equation*}
	\prol \colon 
	(\Cinfty(C),0) 
	\rightleftarrows 
	(\Anti^\bullet \liealg{g}\otimes \Cinfty(M),\del)  
	\colon \iota^*,h
\end{equation*}
is a HE data of the special type of \eqref{eq:simplehedata}, i.e. we
have the following diagram:
\begin{figure*}[h]
	\center
	\begin{tikzcd}[row sep = large, column sep = large]
			0    	&
			\Cinfty(M) \arrow[l]\arrow[d,shift left=-.75ex,swap,"\iota^*"] 
			\arrow[r,shift left=-.75ex,swap,"h_0"]      &		
			\Anti^1\liealg{g}\otimes\Cinfty(M) \arrow[l,shift left=-.75ex,swap,"\del_1"]
			\arrow[r,shift left=-.75ex,swap,"h_1"]  &    
			\cdots \arrow[l,shift left=-.75ex,swap,"\del_2"]  \\
			0     &		 
			\Cinfty(C)  \arrow[l] \arrow[u,shift left=-.75ex,swap,"\prol"]  	&	
			0 \arrow[l]	 	&          	
	\end{tikzcd}
\end{figure*}

For the reduction of equivariant star products, we need to deform it to  
the \emph{quantized Koszul complex}. The \emph{quantized Koszul differential} 
$\boldsymbol{\del} \colon \Anti^\bullet\liealg{g}\otimes\Cinfty(M_\nice)[[\hbar]] 
\longrightarrow \Anti^{\bullet-1}\liealg{g}\otimes\Cinfty(M_\nice)[[\hbar]]$ 
is 
defined by	    
\begin{equation}
  \label{eq:QuantizedKoszulDifferential}
  \boldsymbol{\del}^{(\kappa)}(x \otimes f)
  =  
	 \ins(\basis{e}^a) x \otimes H_a \star f  
  - \frac{\hbar}{2} 
  f^c_{ab} \basis{e}_c \wedge \ins(\basis{e}^a)\ins(\basis{e}^b)x \otimes f
	+ \hbar \kappa f^b_{ab}\ins(\basis{e}^a)
  \left(x\otimes f\right)
\end{equation}
for $\kappa \in \mathbb{C}[[\hbar]]$, 
$x \in \Anti^\bullet \liealg{g}[[\hbar]]$ and 
$f \in \Cinfty(M_\nice)[[\hbar]]$, where $\Delta = f^b_{ab}\basis{e}^a$ is the modular one-form of $\liealg{g}$. 

\begin{remark}
Note that in the literature \cite{bordemann.herbig.waldmann:2000a, 
	gutt.waldmann:2010a} a different convention is used:
  \begin{equation*}
    \boldsymbol{\del}'^{(\kappa)}(x \otimes f)
    =  \ins(\basis{e}^a) x \otimes f \star H_a 
    + \frac{\hbar}{2} 
    f^c_{ab} \basis{e}_c \wedge \ins(\basis{e}^a)\ins(\basis{e}^b)x \otimes f
		+ \hbar \kappa \ins(\Delta)
    \left(x\otimes f\right)
\end{equation*}
for $\kappa \in \mathbb{C}[[\hbar]]$. In particular, 
$\boldsymbol{\del}'^{(\kappa)}$ is left $\star$-linear. However, in order 
to simplify the comparison of the BRST reduction with the reduction via 
$\mathrm{D}_\red$ in Section~\ref{sec:ComparisonofRedProd}, we want 
the quantized Koszul differential to be right $\star$-linear, which leads 
to our convention in \eqref{eq:QuantizedKoszulDifferential}. 
\end{remark}

The reduction of the star product in our convention works analogously 
to \cite{bordemann.herbig.waldmann:2000a,gutt.waldmann:2010a} since 
$\boldsymbol{\del}^{(\kappa)}$ satisfies all 
the desired properties:

\begin{lemma}
  Let $(\star,H)$ be an equivariant star product and 
	$\kappa \in \mathbb{C}[[\hbar]]$. 
	\begin{lemmalist}
	  \item One has $\boldsymbol{\del}^{(0)} \circ \ins(\Delta) + 
	        \ins(\Delta) \circ \boldsymbol{\del}^{(0)} = 0$.
		\item $\boldsymbol{\del}^{(\kappa)}$ is right $\star$-linear.
		\item $\boldsymbol{\del}^{(\kappa)} = \del + O(\hbar)$.
		\item $\boldsymbol{\del}^{(\kappa)}$ is $\group{G}$-equivariant.
		\item One has $ \boldsymbol{\del}^{(\kappa)} \circ 
		      \boldsymbol{\del}^{(\kappa)}    =	0$.
	\end{lemmalist}
\end{lemma}
\begin{proof}
  The proof is analogue to \cite[Lemma~3.4]{gutt.waldmann:2010a}.
\end{proof}

Assume that we have chosen a value $\kappa \in \mathbb{C}[[\hbar]]$ 
and write $\boldsymbol{\del}= \boldsymbol{\del}^{(\kappa)}$. Then by the 
homological perturbation lemma one gets a perturbed homotopy retract
\begin{equation*}
	\begin{tikzcd}[row sep = large, column sep = large]
	0    	&
	\Cinfty(M_\nice)[[\hbar]]\arrow[l]\arrow[d,shift left=-.75ex,
	swap,"\boldsymbol{\iota^*}"] 
	\arrow[r,shift left=-.75ex,swap,"\boldsymbol{h}_0"]      &		
	\Anti^1\liealg{g}\otimes\Cinfty(M_\nice)[[\hbar]] 
	\arrow[l,shift left=-.75ex,swap,"\boldsymbol{\del}_1"]
	\arrow[r,shift left=-.75ex,swap,"\boldsymbol{h}_1"]  &    
	\cdots \arrow[l,shift left=-.75ex,swap,"\boldsymbol{\del}_2"]  \\
	0     &		 
	\Cinfty(C)[[\hbar]]  \arrow[l] 
	\arrow[u,shift left=-.75ex,swap,"\boldsymbol{\prol}"]  	&	
	0, \arrow[l]	 	&          	
	\end{tikzcd}
\end{equation*}
where
\begin{equation}
  \boldsymbol{\prol} 
	= 
	\prol, \hspace{20pt} 
  \boldsymbol{\iota^*}
	= 
	\iota^*(\id + B_1 h_0)^{-1} ,\hspace{20pt} 
  \boldsymbol{h} 
	= 
	h(\id + B h)^{-1},
\end{equation}	
and where $\boldsymbol{\del}-\del =  B$, see 
\eqref{eq:simplifieddeformedhomologicalstuff}. One can show that the 
\emph{deformed restriction map} $\boldsymbol{\iota^*}$ is given by
\begin{equation}
  \boldsymbol{\iota^*} 
  = 
  \iota^* \circ S
  =
  \sum_{r=0} \hbar^r \boldsymbol{\iota^*}_r
  \colon
  \Cinfty(M_\nice)[[\hbar]]
  \longrightarrow 
  \Cinfty(C)[[\hbar]]
\end{equation}
with a $\group{G}$-equivariant formal series of differential operators 
$S= \id+ \sum_{r=1}^\infty \hbar^r S_r$ on  
$\Cinfty(M_\nice)$ and with $S_r$ vanishing on constants. Moreover, it is 
uniquely determined by the properties 
\begin{align}
  \boldsymbol{\iota^*}_0  
  =
  \iota^*,
  \quad
  \boldsymbol{\iota^*} \boldsymbol{\del}_1
  = 
  0
  \quad
  \text{and}
  \quad
  \boldsymbol{\iota^*} \prol
  = 
  \id_{\Cinfty(C)[[\hbar]]}.
\end{align}


The reduced star product $\star_\red$ on $M_\red = C/\group{G}$ is then given by
\begin{equation}
	\label{eq:ReducedStarProduct}
	\pr^*(u_1\star_\red u_2) 
	= 
	\boldsymbol{\iota^*} (\prol(\pr^*u_1)\star \prol(\pr^* u_2))
\end{equation}
for all $u_1,u_2\in\Cinfty(M_\red)[[\hbar]]$, compare 
\cite[Theorem~32]{bordemann.herbig.waldmann:2000a}. 
In \cite[Lemma~4.3.1]{reichert:2017b} it has been shown that 
equivariantly equivalent star products reduce to equivalent star products 
on $M_\red$.

For the comparison of the reduction procedures in 
Section~\ref{sec:ComparisonofRedProd} we need the following observation:

\begin{lemma}
\label{lemma:BoldiotaClassical}
  Let $(\star= \mu + \hbar \pi_\star + \hbar m_G,J)$ be an equivariant 
	star product on $C\times \liealg{g}^*$, and choose $\kappa =- 1$ for 
	the quantized Koszul differential. If one has 
	$\widetilde{\Phi} (\hbar \pi_\star) = 0 = 
	\Phi (\hbar \pi_\star)$, then it follows for all 
	$u_1,u_2\in\Cinfty(M_\red)[[\hbar]]$
  \begin{equation}
  	\pr^*(u_1\star_\red u_2) 
	  = 
	  \boldsymbol{\iota^*} (\prol(\pr^*u_1)\star \prol(\pr^* u_2))
		=
		\iota^* (\prol(\pr^*u_1)\star \prol(\pr^* u_2)).
\end{equation}
\end{lemma}
\begin{proof}
  We have for a polynomial function $ f = P \otimes \phi \in 
	\Sym^j \liealg{g} \otimes \Cinfty(C) \subset \Cinfty(C\times\liealg{g}^*)$
	\begin{align*}
	  (\boldsymbol{\del}-\del) h_0 (P\otimes \phi)
		& =
		\frac{1}{j} \left(\hbar (\pi_\star + m_G) (e_i, \ins(e^i)P\otimes \phi)
		+ \hbar \kappa f^b_{ib} \ins(e^i)P \otimes \phi \right) \\
		& =
		\frac{1}{j}\left(\Phi(\hbar \pi_\star + \hbar m_G) (P\otimes \phi)
		+ \hbar \kappa f^b_{ib} \ins(e^i)P \otimes \phi \right) \\
		& = 
		\frac{1}{j}\left( \hbar m_G (e_i, \ins(e^i) P\otimes \phi)
		+ \hbar \kappa f^b_{ib} \ins(e^i)P \otimes \phi \right)\\
		& =
		\frac{1}{j}\left( \hbar m_{\liealg{g}}(e_i, \ins(e^i)P) \otimes \phi 
		- \ins(e^i) P \otimes \hbar\Lie_{(e_i)_C} \phi  + 
		\hbar \kappa f^b_{ib} \ins(e^i)P \otimes \phi \right),
	\end{align*}
	where $\hbar m_{\liealg{g}}$ denotes the non-trivial part of the 
	Gutt product on $\liealg{g}^*$. We know 
	that $\image (\hbar m_\liealg{g} (e_i, \argument)) \in
	\Sym^{> 0}\liealg{g}[[\hbar]]$, hence it follows 
	\begin{equation}
	\tag{$*$}
	\label{eq:iotaofBh}
	  \iota^* \circ (\boldsymbol{\del}-\del) h_0 (P\otimes \phi)
		=
		\frac{1}{j} \iota^* \left(-\ins(e^i) P \otimes \hbar\Lie_{(e_i)_C} \phi
		+ \hbar \kappa f^b_{ib} \ins(e^i)P \otimes \phi \right).
	\end{equation} 
	On an invariant polynomial $P \otimes \phi \in 
	(\Sym^j \liealg{g} \otimes \Cinfty(C))^\group{G}$ we have
	\begin{equation*}
	  - \ins(e^i) P \otimes \hbar\Lie_{(e_i)_C} \phi 
		=
		-\hbar \ins(e^i) \ad(e_i) P \otimes \phi 
		=
		-\hbar f^i_{ij} \ins(e^j)P \otimes \phi,
	\end{equation*}
	hence \eqref{eq:iotaofBh} vanishes for $\kappa = -1$. 
	Thus we have in this case
	\begin{align*}
  	\pr^*(u_1\star_\red u_2) 
	  & = 
	  \boldsymbol{\iota^*} (\prol(\pr^*u_1)\star \prol(\pr^* u_2)) 
		= 
		\iota^*  (\prol(\pr^*u_1)\star \prol(\pr^* u_2))
  \end{align*}
	and the statement is shown.
\end{proof}

\section{Explicit Formulas for the Homotopy Transfer Theorem}
\label{sec:homotopytransfertheorem}

In is well-known that $L_\infty$-quasi-isomorphisms always admit
$L_\infty$-quasi-inverses. Moreover, it is well-known that given a
homotopy retract one can transfer $L_\infty$-structures, see 
e.g. \cite[Section~10.3]{loday.vallette:2012a}. Explicitly, a homotopy retract (also 
called homotopy equivalence data) consists of two cochain complexes $(A,\D_A)$ and 
$(B,\D_B)$ with chain maps $i,p$ and 
homotopy $h$ such that 
\begin{equation}
  \label{eq:homotopyretract}
  \begin{tikzcd} 
	(A,\D_A)
	\arrow[rr, shift left, "i"] 
  &&   
  \left(B, \D_B\right)
  \arrow[ll, shift left, "p"]
	\arrow[loop right, "h"] 
\end{tikzcd} 
\end{equation}
with $h \circ \D_B + \D_B \circ h = \id - i \circ p$, and such that $i$ and $p$ are
quasi-isomorphisms. Then the homotopy transfer theorem
states that if there exists
a flat $L_\infty$-structure on $B$, then one can transfer it to $A$ in
such a way that $i$ extends to an $L_\infty$-quasi-isomorphism. By 
the invertibility of $L_\infty$-quasi-isomorphisms there also exists an 
$L_\infty$-quasi-isomorphism into $A$ denoted by $P$, see
e.g. \cite[Proposition~10.3.9]{loday.vallette:2012a}.

In this section we state a version of this statement adapted to our applications. 
For simplicity, we assume that we have a deformation retract satisfying
\begin{equation*}
  p\circ i
  =
  \id_A.
\end{equation*}
By \cite[Remark~2.1]{huebschmann:2011a} we can assume that we have even a 
special deformation retract, also called \emph{contraction}, where
\begin{equation*}
  h^2 
  =
  0,
  \quad \quad 
  h \circ i
  =
  0
  \quad \text{ and } \quad 
  p\circ h 
  = 
  0.
\end{equation*} 
Assume now that $(B,Q_B)$ is an $L_\infty$-algebra with $(Q_{B})_1^1=-\D_B$. 
In the following we
give a more explicit description of the transferred $L_\infty$-structure $Q_A$ on $A$ and 
of the $L_\infty$-projection $P\colon (B,Q_B)\rightarrow (A,Q_A)$ inspired by the 
symmetric tensor 
trick \cite{berglund:2014a,huebschmann:2010a,huebschmann:2011a,manetti:2010a}.
The map $h$ extends to a homotopy $H_n
\colon \Sym^n (B[1]) \rightarrow \Sym^n(B[1])[-1]$ with respect to
$Q_{B,n}^n \colon \Sym^n (B[1]) \rightarrow \Sym^n(B[1])[1]$, see
e.g. \cite[p.~383]{loday.vallette:2012a} for the construction on the
tensor algebra, which adapted to our setting works as follows:
we define the operator 
	\begin{align*}
	K_n
	\colon 
	\Sym^n(B[1])
	\longrightarrow 
	\Sym^n(B[1])
	\end{align*}
by 
	\begin{align*}
	K_n(x_1\vee\cdots\vee x_n)
	=
	\frac{1}{n!} 
	\sum_{i=0}^{n-1}
	\sum_{\sigma\in S_n}
	\frac{\epsilon(\sigma)}{n-i}ipX_{\sigma(1)}\vee\cdots\vee ipX_{\sigma(i)}
	\vee X_{\sigma(i+1)}\vee X_{\sigma(n)}.
	\end{align*}
Note that here we sum over the whole symmetric group and 
not the shuffles, since in this case the formulas are easier. We extend $-h$ to a 
coderivation to  $\Sym(B[1])$, i.e.
	\begin{align*}
	\tilde{H}_n(x_1\vee\cdots\vee x_n):=
	-\sum_{\sigma\in \mathrm{Sh}(1,n-1)}
	\epsilon(\sigma) \;
	hx_{\sigma(1)}\vee x_{\sigma(2)}\vee\cdots\vee x_{\sigma(n)}
	\end{align*}
and define 
\begin{align}
 \label{eq:extendedHomotopy}
	H_n=K_n\circ \tilde{H}_n
	=
	\tilde{H}_n\circ K_n.
	\end{align}
Since $i$ and $p$ are chain maps, we have $K_n\circ Q_{B,n}^n=Q_{B,n}^n\circ K_n$,
where $Q_{B,n}^n$ is the extension of the differential $Q_{B,1}^1 = - \D_B$ to 
$\Sym^n(B[1])$ as a coderivation. Hence we have 
	\begin{align*}
	Q_{B,n}^n H_n + H_n Q_{B,n}^n
	=
	(n\cdot\id-ip)\circ K_n,
	\end{align*}
where $ip$ is extended as a coderivation to $\Sym(B[1])$. A combinatorial and 
not very enlightening computation shows that finally 
\begin{equation}
	\label{Eq: ExtHomotopy}
  	Q_{B,n}^n H_n + H_n Q_{B,n}^n
	=
	\id - (ip)^{\vee n}.
\end{equation}
Now assume that we have a codifferential $Q_A$ and a 
morphism of coalgebras $P$ with structure maps  
$P_\ell^1 \colon \Sym^\ell(B[1]) \rightarrow A[1]$ 
such that $P$ is an $L_\infty$-morphism up to order $k$, i.e.
\begin{equation*}
  \sum_{\ell=1}^m P^1_\ell \circ Q_{B,m}^\ell
  =
  \sum_{\ell=1}^m Q_{A,\ell}^1\circ P^\ell_{m}
\end{equation*}
for all $m \leq k$. Then we have the following statement, whose proof can be found in
\cite{esposito.kraft.schnitzer:2020a:pre}.
\begin{lemma}
\label{lemma:Linftyuptok}
  Let $P \colon \Sym(B[1]) \rightarrow \Sym (A[1])$ be an
  $L_\infty$-morphism up to order $k\geq 1$ . Then
  \begin{equation}
    \label{eq:Linftykplusone}
    L_{\infty,k+1}
    =
    \sum_{\ell = 2}^{k+1} Q_{A,\ell}^1 \circ P^\ell_{k+1} - \sum_{\ell =1}^{k} 
    P_\ell^1 \circ 	Q^\ell_{B,k+1}
  \end{equation}
  satisfies
  \begin{equation}
    \label{eq:linftycommuteswithq}
    L_{\infty,k+1} \circ Q_{B,k+1}^{k+1}
    =
    -Q_{A,1}^1 \circ L_{\infty,k+1}.
  \end{equation}
\end{lemma}

This allows us to prove one version of the homotopy transfer theorem.

\begin{theorem}[Homotopy transfer theorem]
  \label{thm:HTTJonas}
  Let $(B,Q_B)$ be a flat $L_\infty$-algebra with differential $(Q_B)^1_1=-\D_B$ and contraction 
	\begin{equation}
    \begin{tikzcd} 
	  (A,\D_A)
	  \arrow[rr, shift left, "i"] 
    &&   
    \left(B, \D_B\right)
    \arrow[ll, shift left, "p"]
	  \arrow[loop right, "h"] .
  \end{tikzcd} 
  \end{equation}
	Then 
	\begin{align}
	  \begin{split}
		\label{eq:HTTTransferredStructures}
		  (Q_A)_1^1
			& =
			- \D_A,
			\quad \quad\quad\quad\;
			(Q_A)^1_{k+1}
			=
			\sum_{i=1}^{k} P^1_i \circ (Q_B)^i_{k+1} \circ i^{\vee (k+1)},\\
			P_1^1
			& =
			p, 
			\quad \quad\quad\quad\quad\quad\;\;
			P^1_{k+1}
			=
			L_{\infty,k+1}\circ H_{k+1}
			\quad \text{ for  } k\geq 1
		\end{split}
	\end{align}
	turns $(A,Q_A)$ into an $L_\infty$-algebra with 
	$L_\infty$-quasi-isomorphism $P\colon (B,Q_B)\rightarrow (A,Q_A)$. 
	Moreover, one has $P^1_k \circ i^{\vee k} =0$ for $k\neq 1$.
\end{theorem}
\begin{proof}
 We observe $P_{k+1}^1(ix_1 \vee \cdots \vee i x_{k+1}) = 0$ for all
 $k\geq 1$ and $x_i \in A$, which directly follows from $h\circ i =
 0$ and thus $H_{k+1} \circ i^{\vee (k+1)} = 0$. Suppose that 
$Q_A$ is a codifferential up to order $k\geq 1$, i.e. 
$\sum_{\ell=1}^m (Q_A)^1_\ell(Q_A)^\ell_m=0$ for all $m \leq k$, 
and that $P$ is an $L_\infty$-morphism 
up to order $k\geq 1$. We know that these conditions are satisfied for $k=1$ and we show that they hold for $k+1$. Starting with $Q_A$ we compute
\begin{align*}
  (Q_AQ_A)^1_{k+1}
	& =
	(Q_AQ_A)^1_{k+1} \circ P^{k+1}_{k+1}\circ i^{\vee (k+1)} 
	=
	\sum_{\ell=1}^{k+1} (Q_AQ_A)^1_\ell P^\ell_{k+1} i^{\vee (k+1)} 
	=
	(Q_AQ_AP)^1_{k+1}i^{\vee (k+1)} \\
	& =
	\sum_{\ell=2}^{k+1} (Q_A)^1_\ell (Q_AP)^\ell_{k+1} i^{\vee (k+1)} 
	+(Q_A)^1_1(Q_AP)^1_{k+1} i^{\vee (k+1)}  \\
	& =
	\sum_{\ell=2}^{k+1} (Q_A)^1_\ell (PQ_B)^\ell_{k+1} i^{\vee (k+1)} 
	+(Q_A)^1_1(Q_A)^1_{k+1}  \\
	& =
	(Q_APQ_B)^1_{k+1}i^{\vee (k+1)}  -(Q_A)^1_1(Q_A)^1_{k+1} + (Q_A)^1_1(Q_A)^1_{k+1} \\
	& =
	\sum_{\ell=1}^{k} (Q_AP)^1_\ell (Q_B)^\ell_{k+1}i^{\vee (k+1)} 
	+ (Q_AP)^1_{k+1} (Q_B)^{k+1}_{k+1}i^{\vee (k+1)} \\
	& = 
	\sum_{\ell=1}^{k} (PQ_B)^1_\ell (Q_B)^\ell_{k+1}i^{\vee (k+1)} 
	+ (Q_AP)^1_{k+1}i^{\vee (k+1)}  (Q_A)^{k+1}_{k+1}\\
	& =
	- (PQ_B)^1_{k+1}i^{\vee (k+1)}  (Q_A)^{k+1}_{k+1} + 
	(Q_AP)^1_{k+1}i^{\vee (k+1)}  (Q_A)^{k+1}_{k+1} \\
	& =
	- (Q_A)^1_{k+1}(Q_A)^{k+1}_{k+1} + 
	(Q_A)^1_{k+1}(Q_A)^{k+1}_{k+1}
	=
	0.
\end{align*}
By the same computation as in Lemma~\ref{lemma:Linftyuptok}, where one in fact only needs 
that $Q_A$ is a codifferential up to order $k+1$, it follows that 
\begin{equation*}
  L_{\infty,k+1} \circ Q_{B,k+1}^{k+1}
  =
  -Q_{A,1}^1 \circ L_{\infty,k+1}.
\end{equation*}
It remains to show that $P$ is an $L_\infty$-morphism
 up to order $k + 1$. We have 
  \begin{align*}
    P_{k+1}^1 \circ (Q_B)^{k+1}_{k+1}
    & = 
    L_{\infty,k+1} \circ H_{k+1} \circ (Q_B)_{k+1}^{k+1} \\
    &=
    L_{\infty,k+1} - L_{\infty,k+1} \circ (Q_B)_{k+1}^{k+1}\circ H_{k+1} 
    -L_{\infty,k+1} \circ (i\circ p)^{\vee (k+1)} \\
    & =
    L_{\infty,k+1} + (Q_A)_1^1 \circ P_{k+1}^1 
  \end{align*}
	since
	\begin{align*}
	  L_{\infty,k+1} \circ (i\circ p)^{\vee (k+1)}
		& =
		\left(\sum_{\ell = 2}^{k+1} Q_{A,\ell}^1 \circ P^\ell_{k+1} - \sum_{\ell =1}^{k} 
    P_\ell^1 \circ 	Q^\ell_{B,k+1}\right) \circ (i\circ p)^{\vee (k+1)}  \\
		& =
		(Q_A)^1_{k+1} \circ p^{\vee (k+1)} 
		- (Q_A)^1_{k+1}\circ p^{\vee(k+1)}
		=
		0.
	\end{align*}
  Therefore
  \begin{equation*}
    P_{k+1}^1 \circ (Q_B)^{k+1}_{k+1} -  (Q_A)_1^1 \circ P_{k+1}^1
    =
    L_{\infty,k+1},
  \end{equation*}
  i.e. $P$ is an $L_\infty$-morphism up to order $k+1$, and the
  statement follows inductively.
\end{proof}
Note that a special case of the above theorem, for $i$ being a DGLA morphism,  
has been proven in \cite[Proposition~3.2]{esposito.kraft.schnitzer:2020a:pre}. 
We also want to give an explicit formula for a $L_\infty$-quasi-inverse of 
$P$, generalizing \cite[Proposition~3.3]{esposito.kraft.schnitzer:2020a:pre}.

\begin{proposition}
  \label{prop:Infinityinclusion}
  The coalgebra map $I\colon \Sym^\bullet (A[1])\to \Sym^\bullet
  (B[1])$ recursively defined by the maps $I_1^1=i$ and $I_{k+1}^1=h\circ
  L_{\infty,k+1}$ for $k\geq 1$ is an $L_\infty$-quasi inverse of $P$.
  Since $h^2=0= h\circ i$, one even has $I_{k+1}^1 = h \circ \sum_{\ell = 2}^{k+1} Q_{B,\ell}^1 \circ I^\ell_{k+1}$ and $P \circ I = \id_A$.
\end{proposition}
\begin{proof}
	We proceed by induction: assume that $I$ is an $L_\infty$-morphism up to 
	order $k$, then we have 
	\begin{align*}
		I^1_{k+1}Q_{A,k+1}^{k+1} - Q_{B,1}^1I^1_{k+1}&= 
		-Q_{B,1}^1\circ h\circ L_{\infty,{k+1}}
		+h\circ L_{\infty,{k+1}}\circ Q_{A,k+1}^{k+1}\\&
		=-Q_{B,1}^1\circ h\circ L_{\infty,{k+1}}-h\circ Q_{B,1}^{1}\circ L_{\infty,{k+1}}\\&
		=(\id-i\circ p)L_{\infty,{k+1}}.
	\end{align*}
	We used that $Q_{B,1}^1=-\D_B$ and the homotopy equation of $h$. 
	Moreover, we get with $p\circ h=0$
	\begin{align*}
	  p \circ L_{\infty,{k+1}}
		& = 
		p\circ  \left( \sum_{\ell = 2}^{k+1} Q_{B,\ell}^1 \circ I^\ell_{k+1} 
		- \sum_{\ell =1}^{k}   I_\ell^1 \circ 	Q^\ell_{A,k+1} \right) \\
		& =
		\sum_{\ell = 2}^{k+1} (P\circ Q_B)^1_\ell \circ I^\ell_{k+1} 
		- \sum_{\ell =2}^{k+1}\sum_{i=2}^{\ell} P^1_i\circ Q^i_{B,\ell} \circ 
		I^\ell_{k+1} - 
		Q^1_{A,k+1} \\
		& =
		\sum_{\ell = 2}^{k+1} (Q_A \circ P)^1_\ell \circ I^\ell_{k+1} 
		- \sum_{i =2}^{k+1}\sum_{\ell=i}^{k+1} P^1_i\circ Q^i_{B,\ell} \circ 
		I^\ell_{k+1} - 		Q^1_{A,k+1} \\
		& =
		Q^1_{A,k+1} - \sum_{i =2}^{k+1}\sum_{\ell=i}^{k+1} 
		P^1_i\circ I^i_{\ell} \circ 	Q^\ell_{A,k+1} - 		Q^1_{A,k+1}
		=
		0,
	\end{align*}
	and therefore $I$ is an $L_\infty$-morphism. 
\end{proof}

\begin{remark}
  Note that in the homotopy transfer theorem the property $h^2=0$ is not 
	needed, and that one can also adapt the above construction of $I$ 
	to this more general case.
\end{remark}

Note that there exists a homotopy equivalence relation $\sim$ between 
$L_\infty$-morphisms, see e.g. \cite{dolgushev:2007a} such that 
equivalent $L_\infty$-morphisms map Maurer-Cartan elements to equivalent 
Maurer-Cartan elements, see e.g. 
\cite[Lemma~B.5]{bursztyn.dolgushev.waldmann:2012a} for the case of DGLAs 
and \cite[Proposition~1.4.6]{kraft:2021a} for the case of flat 
$L_\infty$-algebras. Then we get:

\begin{corollary}
\label{cor:IPsimId}
  In the above setting one has $ P \circ I = \id_A$ and 
	$I \circ P \sim \id_B$. In particular, assume that one has complete 
	descending filtrations on $A,B$ such that all the maps are compatible. 
	Then every Maurer-Cartan element $\pi\in \mathcal{F}^1B$ 
	is equivalent to $(I\circ P)^1(\cc{\exp}(\pi))$.
\end{corollary}
\begin{proof}
By \cite[Proposition~3.8]{kraft.schnitzer:2021a:pre} $P$ admits a quasi-inverse $I'$ such that $P \circ I' \sim \id_A$ and $I'\circ P\sim\id_B$, 
which implies
\begin{equation*}
  I \circ P 
	=
	\id_B \circ I \circ P
	\sim 
	I' \circ P \circ I \circ P 
	=
	I' \circ P
	\sim 
	\id_B,
\end{equation*}
the rest of the statement is then clear.
\end{proof}

%
%
\bibliographystyle{nchairx}

\end{document}